\documentclass[12pt]{amsart}
\usepackage[margin=1.4in]{geometry}

\usepackage[utf8]{inputenc}
\usepackage{mathtools} 
\usepackage{physics}
\usepackage{amsmath}
\usepackage{amsthm}
\usepackage{amssymb}
\usepackage{hyperref}
\usepackage[all]{xy}
\usepackage{tikz-cd}
\usepackage{bbold}
\usepackage{colortbl}
\usepackage{nicematrix}
\usepackage{stmaryrd}
\usepackage{mathrsfs} 
\usepackage[normalem]{ulem} 
\usepackage{dsfont}
\usepackage{url}
\usepackage{graphicx}
\usepackage{wrapfig}

\newtheorem{theorem}{Theorem}[subsection]

\newtheorem{lemma}[theorem]{Lemma}
\newtheorem{corollary}[theorem]{Corollary}

\newtheorem*{theorem*}{Theorem}

\theoremstyle{definition}
\newtheorem{definition}[theorem]{Definition}
\newtheorem{example}[theorem]{Example}
\newtheorem{construction}[theorem]{Construction}

\newtheorem*{definition*}{Definition}

\theoremstyle{remark}
\newtheorem{remark}[theorem]{Remark}

\makeatletter
\let\c@equation\c@theorem

\makeatother

\newcommand{\g}{\mathfrak{g}}
\newcommand{\kompact}{\mathfrak{k}}

\newcommand{\ab}{\mathfrak{a}}

\newcommand{\p}{\mathfrak{p}}
\newcommand{\m}{\mathfrak{m}}
\newcommand{\n}{\mathfrak{n}}

\newcommand{\R}{\mathbb{R}}
\newcommand{\C}{\mathbb{C}}
\newcommand{\N}{\mathbb{N}}

\newcommand{\HS}{\mathcal{H}}

\newcommand{\A}{\mathcal{A}}
\newcommand{\E}{\mathcal{E}}
\newcommand{\F}{\mathcal{F}}

\newcommand{\ad}{\operatorname{ad}}

\newcommand{\cfunc}{\mathbf{c}}

\newcommand{\HIwasawa}{\boldsymbol{H}}
\newcommand{\kIwasawa}{\boldsymbol{k}}
\newcommand{\nIwasawa}{\boldsymbol{n}}

\newcommand{\GmodmodK}{K\operatorname{\setminus} G\operatorname{/} K}

\newcommand{\supp}{\operatorname{supp}}
\newcommand{\smooth}{C^\infty}
\newcommand{\csmooth}{C_c^\infty}
\newcommand{\schwartz}{\mathscr{S}}
\newcommand{\HCschwartz}{\mathscr{C}}

\newcommand{\Circled}[1]{%
  \tikz[baseline=(X.base)] 
    \node[draw, circle, inner sep=1pt](X){#1};%
}

\title[Equivariant PSDOs on Noncompact Symmetric Spaces]{Complete Symbols of Equivariant Pseudodifferential Operators on Noncompact Symmetric Spaces}
\author{Satwata Hans}
\address{Department of Mathematics, The Pennsylvania State University, University Park, PA 16802.}
\email{hans@psu.edu}
\date{}

\begin{document}

\begin{abstract}
    We study $G$-equivariant Hörmander pseudodifferential operators on a noncompact symmetric space $G/K$. We define a notion of a \emph{complete symbol function}, called the \emph{Harish-Chandra symbol function}, on the spherical tempered dual of $G$, for operators that satisfy a rapid off-diagonal decay condition on their Schwartz kernels, and we characterize the class of such operators in terms of their Harish-Chandra symbol function.
\end{abstract}

\maketitle

\section{Introduction}
At the heart of the theory of pseudodifferential operators on $\R^n$, as studied by Kohn and Nirenberg in \cite{Kohn-Nirenberg65}, and accounted for in detail by Hörmander in \cite[Chapter XVIII]{Hormander-III} is the notion of a \emph{complete symbol function} $a(x,\xi)\in \smooth(\R^n\times \R^n)$, see \eqref{eqn-symbol-bound-Rn}, which defines a \emph{pseudodifferential operator (PSDO)} $T_a:\csmooth(\R^n)\rightarrow \smooth(\R^n)$ via the Fourier inversion formula on $\R^n$ i.e., \begin{equation*}
    (T_a f)(x) \coloneqq \int_{\R^n} a(x,\xi)\,\hat{f}(\xi)\,e^{i x\cdot \xi} \, d\xi\,.
\end{equation*} It turns out that pseudodifferential operators are invariant under diffeomorphisms, which lets one extend the theory of PSDOs to a smooth manifold $M$ by defining the operator on local charts. These operators are called \emph{Hörmander pseudodifferential operators on $M$}. However, they cease to have a meaningful notion of a complete symbol: the only version of the symbol function that survives the leap from Euclidean spaces to smooth manifolds is the \emph{principal symbol} $\sigma:T^*M\setminus\{0\}\to \C$, which encodes the ``leading order'' behavior of a Hörmander PSDO.

The non-existence of a complete symbol function for a Hörmander pseudodifferential operator can be linked to the non-existence of a suitable Fourier theory on an arbitrary smooth manifold $M$. So, in this paper we shall investigate the following question: if $M$ is a homogeneous space for a Lie group $G$, does the representation theory of $G$ gives rise to a meaningful notion of a complete symbol for Hörmander PSDOs on $M$?  

For \emph{compact} Lie groups and compact homogeneous spaces, such a study has been pursued extensively, as can be found in the book \cite{RuzhanskyTurunen2010} by Ruzhansky and Turunen. On the other hand, for a \emph{noncompact} Lie group $G$ (and noncompact homogeneous spaces), the question has remained mostly unexplored, perhaps due to the more nuanced theory of tempered unitary representations of $G$. In a recent paper \cite{DebelloHigson}, DeBello and Higson initiate a study along this line. They consider the setting of noncompact semisimple Lie groups $G$ of \emph{real rank 1} with a maximal compact subgroup $K$, and study the class of \emph{properly-supported $G$-equivariant Hörmander pseudodifferential operators} on the symmetric space $G/K$ (and even on homogeneous vector bundles over $G/K$). They show that every $G$-equivariant Hörmander PSDO has a principal symbol that is a homogeneous function on the tempered dual of $G$, and they characterize these operators, in $C^*$-algebraic terms, through this principal symbol, for groups of real rank 1; see \cite[Theorem 6.5.2]{DebelloHigson}. 

Following them, an obvious question is whether their characterization can be lifted beyond the principal symbol, and specifically, is there a notion of a \emph{representation-theoretic complete symbol} for the class of $G$-equivariant Hörmander PSDOs on noncompact symmetric spaces (indeed, also for higher real rank)? The author is aware of the same question also being raised elsewhere, for instance \cite{Stanton-Tomas78} and \cite{Duistermaat-Iwasawa-projection}.

In this paper, the question of the existence of a representation-theoretic complete symbol for the class of properly-supported, $G$-equivariant Hörmander PSDOs on a noncompact Riemannian symmetric space is answered in the affirmative. In addition, the same conclusion is reached also for $G$-equivariant Hörmander PSDOs (not necessarily properly-supported) whose kernels satisfy a rapid-decay property away from the diagonal. In order to prove this, the notion of a \emph{Harish-Chandra symbol function} (see Definition \ref{defn-harish-chandra-symbol}) $m\in \smooth(\ab^*)^W$ is introduced, which is a function on the \emph{spherical tempered dual} of $G$ that satisfies typical symbol-type estimates. Then the Plancherel formula for $L^2(G/K)$ is used to define the notion of \emph{$G$-equivariant Harish-Chandra pseudodifferential operators} (see Definition \ref{defn-harish-chandra-PSDO}), and the following theorem is proved. See Theorem \ref{thm-Hormander-implies-HC-upgrade} and Theorem \ref{thm-HC-implies-hormander} for more precise versions.
\begin{theorem*}
    A $G$-equivariant continuous linear operator $$T:\csmooth(G/K)\to \csmooth(G/K)$$ is a $G$-equivariant Harish-Chandra pseudodifferential operator if and only if it is a $G$-equivariant Hörmander pseudodifferential operator whose Schwartz kernel is a Harish-Chandra Schwartz function ``away from the diagonal". 
    \footnote{The theorem proved here is without any vector bundles unlike \cite{DebelloHigson}. The general case with vector bundles is more complicated, and will require a more in-depth study of the tempered dual of $G$.}
\end{theorem*}

In essence, the Harish-Chandra symbol function (see Definition \ref{defn-harish-chandra-symbol}) of a $G$-equivariant pseudodifferential operator \emph{serves} as a representation-theoretic complete symbol.  Proving this, however, involves an analysis of the Schwartz kernels of $G$-equivariant Harish-Chandra PSDOs, which is heavily inspired by the treatment of kernels of Euclidean pseudodifferential operators in \cite[Chapter VI.4]{Stein93-Real-Variable-Methods}. A brief sketch of this principle will be provided in Section \ref{subsec-kernel-criteria-euclidean}. In particular, the following theorem about kernels of $G$-equivariant Harish-Chandra pseudodifferential operators is proved here. See Theorem \ref{thm-kernel-estimates-G/K} for a more precise statement. 

\begin{theorem*}
    Let $m(\lambda)$ be a Harish-Chandra symbol function of order $r$, and let $T_m$ be the associated $G$-equivariant Harish-Chandra pseudodifferential operator. The kernel $k(\cdot)$ in the expression $$(T_m f)(x)=\int_{G/K} k(y^{-1}x)\,f(y)\,dy$$ is a smooth function on $G/K$ away from $eK$, and satisfies the estimates \begin{equation*}
        |k(D_1;z;D_2)|\lesssim_{N,D_1,D_2}\, \varphi_0(z)\cdot |z|^{-n-r-d_1-d_2-N}
    \end{equation*} for all $N\geq 0$ such that $n+r+d_1+d_2+N>0$, where $n=\dim(G/K)$ and $D_i$'s are invariant differential operators of order $d_i$.
\end{theorem*}

Although the terminology of ``$G$-equivariant Harish-Chandra pseudodifferential operators'' is used for the first time in this paper, these are essentially multiplier operators on noncompact symmetric spaces, and multiplier operators with symbol-like properties have been extensively studied in regards to $L^p$-boundedness theory. These appear in the works of Clerc and Stein in \cite{ClercStein74} for complex semisimple Lie groups, Stanton and Tomas in \cite{Stanton-Tomas78} for groups of real rank 1, and Anker in \cite{Anker90} for all real semisimple Lie groups. Anker uses a result of Harish-Chandra involving the Abel transform to convert questions involving spherical Fourier analysis on the symmetric space to questions on Euclidean Fourier analysis on a Euclidean space (see the diagram in \eqref{diagram-Abel-triangle}). This is critical because oscillatory integrals on symmetric spaces require new estimation techniques, and Anker's approach (see Construction \ref{Chopping - 2}) does the trick.\\

\paragraph{\textbf{Acknowledgments}} I thank my advisor, Nigel Higson, for suggesting this project, and for his constant guidance, help and regular doses of motivation throughout this investigation. I would also like to thank Jean-Philippe Anker for suggesting techniques that have been crucial in this work, and Peter Hintz and Anna Mazzucato for helpful discussions. The author acknowledges the support of the Institut Henri Poincaré (UAR 839 CNRS-Sorbonne Université) and LabEx CARMIN (ANR-10-LABX-59-01).  The author was also partially supported  by the NSF grant DMS-1952669.

\section{Analysis on Noncompact Symmetric Spaces}

The purpose of this section is to provide a brief summary of the necessary background on the structure theory of noncompact semisimple Lie groups, and harmonic analysis on the corresponding noncompact symmetric spaces. The standard references for most of the material include \cite{HelgasonGGA}, \cite{HelgasonGASS} and \cite{Knapp-example}.

\subsection{Structure theory of $G$}\label{sec-structure-of-G}
Let $G$ be a noncompact connected real semisimple Lie group, with finite center, and $K$ be a maximal compact subgroup. The corresponding homogeneous space $G/K$ (the set of left cosets of $K$ in $G$) is a Riemannian symmetric space of noncompact type. Let $\g$ be the Lie algebra of $G$, with Cartan decomposition $\g=\kompact\oplus\p$, where $\kompact=\operatorname{Lie}(K)$. There is a natural identification between the vector space $\p$ and the tangent space of $G/K$ at the base point $eK$. The Killing form on $\g$ defined by $B(X,Y)=\Tr(\ad_X \ad_Y)$ for $X,Y\in \g$ restricts to a $K$-invariant inner product on $\p$, and hence, induces a $G$-invariant Riemannian metric on $G/K$. Moreover, the exponential map $\exp:\g\rightarrow G$ induces a diffeomorphism between $\p$ and $G/K$. Inheriting the ``distance from origin'' on $\p$, one defines the \emph{distance from $eK$} on $G/K$ by setting $|\exp(X)|_{G/K}\coloneqq |X|_{\p}$. The subscripts will be ignored in general, since it will be clear from the context which one is meant.

Let us fix a maximal abelian subspace $\ab\subset \p$. Any two maximal abelian subspaces are conjugate by an element of $K$, and their dimension is called the \emph{real rank} of $G$, and will be denoted as $a=\dim \ab$. Diagonalizing the family $\{\ad_H:\g\rightarrow\g\,:\,H\in \ab\}$ gives rise to a decomposition of $\g$ into simultaneous eigenspaces $\g_\alpha$, called the \emph{restricted root space decomposition} of $(\g,\ab)$, with eigenvalues $\alpha\in \Sigma(\g,\ab)\eqqcolon \Sigma$ called the \emph{restricted roots} i.e., $$\g=\m\oplus \ab\oplus \bigoplus_{\alpha\neq 0}\g_\alpha$$ where $\m$ is the centralizer of $\ab$ in $\kompact$. We also introduce the subgroup $M=Z_K(\ab)$, the centralizer of $\ab$ in $K$. Now, let us make a choice of positive roots $\Sigma^+$, and let $\Sigma_0^+$ be its associated set of indivisible roots. The positive roots determine a positive Weyl chamber $\ab^+\coloneqq \{H\in \ab\,:\,\alpha(H)>0 \text{ for all }\alpha\in \Sigma^+\}$, and let $\overline{\ab^+}$ be its closure in $\ab$. Define the Weyl group as $W\coloneqq N_K(\ab)/Z_K(\ab)$, where $`N$' refers to the normalizer.

Let $\n=\bigoplus_{\alpha\in \Sigma^+}\g_\alpha$ and let $N$ be the analytic subgroup of $G$ corresponding to $\n$, and let $A=\exp \ab$. Then $G$ admits an \emph{Iwasawa decomposition} $G=KAN$. The Iwasawa decomposition gives rise to ``projection'' maps \begin{equation}\label{eqn-Iwasawa-maps}
    \kIwasawa: G\rightarrow K, \quad \HIwasawa: G\rightarrow \ab, \quad \nIwasawa: G\rightarrow N
\end{equation} such that $g=\kIwasawa(g)\cdot \exp(\HIwasawa(g))\cdot \nIwasawa(g)$ for all $g\in G$. The group $G$ also admits another (non-unique) decomposition $G=KAK$, called the \emph{KAK decomposition}, which is particularly useful when integrating bi-$K$-invariant functions on $G$; the corresponding integration formulae will be introduced in \eqref{eqn-jacobian-KAK}.

\subsection{Spherical analysis on $G/K$ and the Plancherel formula}

The space $G/K$ admits a natural $G$-invariant measure on $G/K$ induced from the bi-invariant Haar measure $dg$ on $G$. Hence, $L^2(G/K)$ with the left-regular representation is a unitary representation of $G$. It may also be viewed as the closed subspace of $L^2(G)$ where $K$ acts trivially on the right. And, the irreducible tempered unitary representations of $G$ appearing in the Plancherel decomposition of $L^2(G/K)$ are precisely those that admit a nonzero $K$-fixed vector. 

These representations constitute the \emph{spherical tempered dual} of $G$, and are parameterized by $\lambda\in \ab^*$. In particular, each element in the spherical principal series is induced from the minimal parabolic subgroup $P=MAN$ of $G$ via parabolic induction, defined by \begin{equation}\label{eqn-spherical-representation}
\pi_\lambda\coloneqq \operatorname{Ind}_{MAN}^G(1\otimes e^{i\lambda+\rho}\otimes 1)\,,
\end{equation} where $\rho=\frac{1}{2}\sum_{\alpha>0}(\dim \g_\alpha)\cdot \alpha \in \ab^*$. The spherical representations $\pi_\lambda$ are unitary and irreducible for all $\lambda\in \ab^*$, and $\pi_{\lambda_1}$ is unitarily equivalent to $\pi_{\lambda_2}$ if and only if $\lambda_1$ and $\lambda_2$ are in the same $W$-orbit (the $W$-action on $\ab^*$ is induced from the $W$-action on $\ab$ through the Killing form). And, the Plancherel formula for $L^2(G/K)$ is given by \begin{equation}\label{eqn-Plancherel-G/K}
    f(x)=\frac{1}{|W|}\int_{\ab^*} \Tr(\pi_\lambda(x^{-1})\pi_\lambda(f))\,|\cfunc(\lambda)|^{-2}\,d\lambda \quad \text{ for }x\in G/K
\end{equation}
where $|\cfunc(\lambda)|^{-2}\,d\lambda$ is the Plancherel measure of $G$ on the spherical tempered dual. The function $\cfunc$ is the famous \emph{Harish-Chandra $\cfunc$-function}, which satisfies the estimates \begin{equation}\label{eqn-c-function-estimate}
|\cfunc(\lambda)|^{-2}\lesssim \langle\lambda\rangle^{n-a} \,,
\end{equation} where $n=\dim G/K$, $a=\dim \ab$, and $\langle\lambda\rangle=(1+|\lambda|^2)^{1/2}$.

\subsection{The Helgason-Fourier transform}

Through the compact picture (see \cite[Chapter VII.1]{Knapp-example}), $\pi_\lambda$ can be realized on the Hilbert space $L^2(K/M)$ with the $G$-action \begin{equation}
    (\pi_\lambda(g)f)(k)=e^{-(i\lambda+\rho)(\HIwasawa(g^{-1}k))}\cdot f(\kIwasawa(g^{-1}k))
\end{equation} where $\kIwasawa$ and $\HIwasawa$ are the maps arising from the Iwasawa decomposition in \eqref{eqn-Iwasawa-maps}. Moreover, the space of $K$-fixed vectors in $\pi_\lambda$ is $1$-dimensional, and the trace of the operator $\pi_\lambda(f)$ restricted to the $K$-fixed vectors is given by \begin{equation}
    \langle \pi_\lambda(f)1, 1 \rangle_{L^2(K/M)}=\int_{K/M}\int_{G/K}\,f(x)\,e^{-(i\lambda+\rho)(\HIwasawa(x^{-1}k))}\,dk\,.
\end{equation}This motivates the following definition. 

\begin{definition}[Helgason-Fourier transform]\label{def-helgason-fourier-transform}
    The \emph{Helgason-Fourier transform} of a function $f\in \csmooth(G/K)$ is defined as $$\widetilde{f}(\lambda,k)\coloneqq\int_{G/K}f(x)\,e_{\lambda,k}(x)\,dx=\int_{G} f(g)\,e_{\lambda,k}(g)\,dg $$ for all $(\lambda,k)\in \ab^*\times K/M$, where $e_{\lambda,k}(x)=e^{-(i\lambda+\rho)(\HIwasawa(x^{-1}k))}$. \footnote{The notation $e_{\lambda,k}$ is inspired by \cite{Guenda-24-Paley-Wiener}.} 
\end{definition}

The Helgason-Fourier transform will serve as the ``natural lens" through which we should look at $G$-equivariant operators on $G/K$. Moreover, the Plancherel formula \eqref{eqn-Plancherel-G/K} can be re-stated in terms of the Helgason-Fourier transform as \begin{align}\label{eqn-Helgason-Fourier-inversion}
    \begin{split}
    f(x)=\frac{1}{|W|}\int_{\ab^*}\int_{K/M}\,\widetilde{f}(\lambda,k)\,e_{-\lambda,k}(x)\,dk\,|\cfunc(\lambda)|^{-2}\,d\lambda\,.
    \end{split}
\end{align} This will be used to define the notion of $G$-equivariant Harish-Chandra pseudodifferential operators in Definition \ref{defn-harish-chandra-PSDO}.

The Helgason-Fourier transform has a simpler description when restricted to bi-$K$-invariant functions on $G$, where it reduces to the \emph{spherical transform} (also called the \emph{Harish-Chandra transform}), which for $f\in \csmooth(\GmodmodK)$ is defined as \begin{equation}\label{eqn-spherical-transform-defn}
\HS(f)(\lambda)\coloneqq\int_{G/K}f(x)\,\varphi_{\lambda}(x)\,dx = \int_G f(g)\,\varphi_{\lambda}(g)\,dg
\end{equation} where $\varphi_\lambda(x)=\int_K e_{\lambda,k}(x)\,dk$ is the \emph{spherical function} associated to $\lambda\in \ab^*$. By definition, the spherical functions are bi-$K$-invariant on $G$, and they satisfy the symmetry properties 
\begin{equation}\label{eqn-spherical-function-symmetry}
\varphi_\lambda(x)=\varphi_{-\lambda}(x^{-1}), \quad \text{and} \quad \varphi_{w.\lambda}=\varphi_\lambda \quad \text{for all } w\in W\,.
\end{equation} Analytic properties of the spherical functions will be quite important in this paper, and the following lemma lists some of them.  
\begin{lemma}[{\cite[Chapter III, Lemma 1.18]{HelgasonGASS}}]\label{lem-spherical-functions} For $\lambda\in \ab^*$, let $\varphi_\lambda$ be the associated spherical function.
    \begin{enumerate}
        \item The spherical functions are left $K$-invariant joint eigenfunctions of the algebra of $G$-invariant differential operators on $G/K$. Moreover, $$\Delta \varphi_\lambda=-(|\lambda|^2+|\rho|^2)\cdot\varphi_\lambda$$ for the Laplace-Beltrami operator $\Delta$ on $G/K$.
        \item For $D_1,D_2\in U(\g)$,  $|\varphi_\lambda(D_1;\,x\,;D_2)|\leq C_{D_1,D_2}\cdot \langle \lambda\rangle^{\deg D_1+\deg D_2}\cdot\varphi_0(x)$ for some constant $C_{D_1,D_2}>0$. \footnote{$U(\g)$ is the universal enveloping algebra of $\g_\C\cong \g\otimes_\R \C$. $f(D_1;x;D_2)$ is Harish-Chandra's notation for $D_1,D_2\in U(\g)$ acting on $f$ as right-invariant and left-invariant differential operators respectively, and their actions commute. See \cite[Chapter III]{Knapp-example}.}
        \item For any multi-index $\alpha$, there exists a constant $C_\alpha>0$ such that $|\partial_\lambda^\alpha \varphi_\lambda(x)|\leq C_\alpha \cdot |x|^{|\alpha|}\cdot \varphi_0(x)$ for all $\lambda\in \ab^*$ and $x\in G/K$. 
    \end{enumerate}
\end{lemma} Similar to \eqref{eqn-Helgason-Fourier-inversion}, the inverse spherical transform is given by \begin{equation}\label{eqn-spherical-transform-inversion}
    f(x)=\frac{1}{|W|}\int_{\ab^*}\HS(f)(\lambda)\,\varphi_{-\lambda}(x)\,|\cfunc(\lambda)|^{-2}\,d\lambda\,.
\end{equation} 
Helgason proved a spherical Paley-Wiener theorem (see \cite[Chapter IV, Theorem 7.1]{HelgasonGGA}), establishing that the spherical transform $\HS$ is a topological isomorphism between $\csmooth(\GmodmodK)$ and $\operatorname{PW}(\ab^*_\C)^W$, the space of $W$-invariant holomorphic functions on $\ab^*_\C$ ($\cong \ab^*\otimes_\R \C$) that satisfy the Euclidean Paley-Wiener conditions (see \cite[Theorem 7.3.1]{Hormander-I}). For more details on the transforms, especially regarding the Helgason-Fourier transform, see \cite[Chapter III]{HelgasonGASS}. 

\subsection{The Abel transform and Euclidean analysis on $\ab$}

The Abel transform will act as the bridge connecting the spherical transform \eqref{eqn-spherical-transform-defn} to the Euclidean Fourier transform. The Abel transform $\A:\csmooth(\GmodmodK)\rightarrow \csmooth(\ab)^W$ is defined by \begin{equation}\label{eqn-Abel-defn}
    \mathcal{A}f(H)=e^{\rho(H)}\int_N f(\exp(H)n)\,dn \quad \text{ for }H\in \ab
\end{equation} for $f\in \csmooth(\GmodmodK)$. The fact that $\A f$ is $W$-invariant is a result of Harish-Chandra (see \cite[Chapter I, Theorem 5.7]{HelgasonGGA}). A direct calculation gives $$\A=\mathcal{F}^{-1}\circ\mathcal{I} \circ \HS$$ where $\mathcal{F}$ is the Euclidean Fourier transform on $\ab$ defined by $$\mathcal{F}f(\lambda)=\int_{\ab} f(H)e^{-i\lambda(H)}\,dH\quad \text{ for }f\in \csmooth(\ab)^W$$ and $\mathcal{I}(\psi)(\lambda)=\psi(-\lambda)$ for $\psi:\ab^*\rightarrow \C$. This can be summarized in the following commutative diagram \footnote{Ideally, this should be a triangle. However, conflicting traditions for sign conventions used in $\F$ and in $\HS$ force us to use the $\mathcal{I}$ map.}, where the left arm represents spherical analysis and the right arm represents Euclidean analysis. \begin{equation}\label{diagram-Abel-triangle}
    \begin{tikzcd}
         \operatorname{PW}(\ab^*_\C)^W \arrow[r,"\mathcal{I}","\cong"'] &  \operatorname{PW}(\ab^*_\C)^W\\ \csmooth(\GmodmodK) \arrow[u, "\HS", "\cong"'] \arrow[r,"\cong","\A"']&  \csmooth(\ab)^W \arrow[u,"\cong", "\mathcal{F}"']  
    \end{tikzcd}
\end{equation} 
The following theorem uses the spherical Paley-Wiener theorem of Helgason along with the classical Paley-Wiener theorem for the Euclidean Fourier transform.

\begin{theorem}[{\cite[Proposition 5]{Anker91}}]\label{thm-abel-support-preserving}
    The Abel transform $$\mathcal{A}:\csmooth(\GmodmodK)\longrightarrow \csmooth(\ab)^W$$ is a topological isomorphism and preserves supports in the following sense: for
    $f\in \csmooth(\GmodmodK)$, $f$ has support contained in the ball $\{x\in G\,:\,|x|\leq R\}$ if and only if $\mathcal{A}f$ has support contained in the ball $\{H\in \ab\,:\,|H|\leq R\}$.
\end{theorem}

The previous theorem can also be stated for compactly supported distributions by duality. Now, the following corollary of this support-preserving property is at the heart of Anker's method in \cite{Anker92_Sharp_Estimates_Laplacian}.

\begin{corollary}\label{cor-abel-support-invariance}
    Let $x\in G/K$, and let $f\in \csmooth(G/K)$. If there exists a function $\Omega_{x}\in \csmooth(\ab)^W$ such that $\Omega_{x}(H)=1$ for all $|H|>\frac{3|x|}{4}$ and $\Omega_{x}(H)=0$ for all $|H|<\frac{|x|}{2}$. Then \begin{equation*}
    f(x)=\A^{-1}(\Omega_{x}\cdot \mathcal{A}f)(x)\,.
    \end{equation*}
\end{corollary}
When studying the pointwise behavior of a bi-$K$-invariant function (or a distribution) on $G$ away from the base point $e$, this corollary allows one to replace the function (or the distribution) with another that agrees with the original at the point in question, but whose Abel transform is supported away from the origin in $\ab$. This will be an essential tool in the proof of Theorem \ref{thm-kernel-estimates-G/K}.

\subsection{The bi-$K$-invariant Harish-Chandra Schwartz space}\label{subsec-schwartz-space} This section will provide the appropriate notion of \emph{rapid-decay} for functions on $G$ and $G/K$.

The \emph{Harish-Chandra Schwartz space} $\HCschwartz(G)$ consists of all functions $f\in\smooth(G)$ such that for all $q\in \N\cup\{0\}$ and $D_1,D_2\in U(\g)$, the seminorms \begin{equation}\label{eqn-HC-schwartz-seminorms}
    \|f\|_{D_1,D_2,q}\coloneqq\sup_{g\in G} (1+|g|)^q \cdot \varphi_0(g)^{-1}\cdot |f(D_1;g;D_2)|
\end{equation} are finite. $\HCschwartz(G)$ is a Fréchet algebra under convolution, with its topology induced by the seminorms in \eqref{eqn-HC-schwartz-seminorms}. See \cite{ArthurRankOne70} for more details on $\HCschwartz(G)$. The \emph{bi-$K$-invariant Harish-Chandra Schwartz space} $\HCschwartz(\GmodmodK)$ consists of the elements of $\HCschwartz(G)$ that are bi-$K$-invariant. Furthermore, if $\schwartz(\ab^*)^W$ is the space of $W$-invariant Euclidean Schwartz functions on $\ab^*$, then one has the following theorem.
\begin{theorem}\label{thm-Schwartz-algebra-characterization}{\cite[Chapter III, Theorem 1.17]{HelgasonGASS}}
    The bi-$K$-invariant Harish-Chandra Schwartz space $\HCschwartz(\GmodmodK)$ is a Fréchet algebra under convolution, and $\mathcal{H}:\HCschwartz(\GmodmodK)\xrightarrow{\cong} \schwartz(\ab^*)^W$ is a Fréchet algebra isomorphism with respect to the pointwise product on $\schwartz(\ab^*)^W$.
\end{theorem}The space $\csmooth(\GmodmodK)$ is dense in $\HCschwartz(\GmodmodK)$ (see \cite[Chapter III, Lemma 1.21]{HelgasonGASS}), and the Abel transform $\A$ in \eqref{eqn-Abel-defn} extends to a topological isomorphism $\A:\HCschwartz(\GmodmodK)\xrightarrow{\cong} \schwartz(\ab)^W$. Furthermore, one has an analogous commutative diagram to that of \eqref{diagram-Abel-triangle} with all the spaces replaced by their Schwartz counterparts.

Let $\HCschwartz'(\GmodmodK)$ be the topological dual of $\HCschwartz(\GmodmodK)$ consisting of bi-$K$-invariant elements, called the \emph{bi-$K$-invariant tempered distributions on $G$}. Of course, the spherical transform in \eqref{eqn-spherical-transform-defn} induces a topological isomorphism $\mathcal{H}'\,:\,\schwartz'(\ab^*)^W\xrightarrow{\cong} \HCschwartz'(\GmodmodK)$. In particular, there is a convolution product between the tempered distributions $\HCschwartz'(\GmodmodK)$ and $\HCschwartz(\GmodmodK)$. Let $u\in \HCschwartz'(\GmodmodK)$ and $f\in \HCschwartz(\GmodmodK)$, then we can define $u\ast f\in \HCschwartz'(\GmodmodK)$ as a smooth function by \begin{equation}
    (u\ast f)\coloneqq \HS^{-1}\big(\HS(f)\cdot \HS(u)\big)\,.
\end{equation} See \cite[Chapter III, \S 1.2]{HelgasonGASS} for detailed statements and proofs of these facts. Elements of $\HCschwartz'(\GmodmodK)$, with some additional properties, will turn out to be the kernels of $G$-equivariant Harish-Chandra pseudodifferential operators (see Theorem \ref{thm-kernel-estimates-G/K}), and a generalized notion of this convolution product will be used to define the action of an equivariant Harish-Chandra pseudodifferential operator on $\HCschwartz(G/K)$.

\section{Pseudodifferential Operators (PSDOs)} The purpose of this section is to briefly discuss the theory of pseudodifferential operators on Euclidean spaces, and Hörmander pseudodifferential operators on smooth manifolds. The standard references for the material in this section include \cite{Hormander-III} and \cite{Stein93-Real-Variable-Methods}.

\subsection{PSDOs on Euclidean spaces}
On $\R^n$, a \emph{symbol function} $a(x,\xi)$ of \emph{order $r$} is a smooth function of $(x,\xi)\in \R^n\times \R^n$ that satisfies the estimates \begin{equation}\label{eqn-symbol-bound-Rn}
    |\partial_x^\beta \partial_\xi^\alpha a(x,\xi)|\leq C_{\alpha,\beta}\cdot \langle \xi\rangle^{r-|\alpha|} 
\end{equation} for all $x\in \R^n$ and $|\xi|\geq 1$. The set of symbol functions of order $r$ is denoted by $S^r(\R^n)$, and the corresponding \emph{pseudodifferential operator (PSDO)} $T_a:\csmooth(\R^n)\rightarrow \smooth(\R^n)$ of \emph{order $r$} is defined as \begin{equation}\label{eqn-PSDO-Rn}
    (T_a f)(x) \coloneqq \int_{\R^n} a(x,\xi)\, \widehat{f}(\xi)\, e^{i x \cdot \xi}\, d\xi\,.
\end{equation} This is a continuous operator from $\csmooth(\R^n)\rightarrow \smooth(\R^n)$ and from $\schwartz(\R^n)\rightarrow \schwartz(\R^n)$ \cite[Theorem 18.1.6]{Hormander-III}, and the translation-invariant ones are precisely the Fourier multipliers with smooth symbols i.e., those for which $a(x,\xi)=m(\xi)$ is independent of $x$. 

\subsection{The kernel criteria for PSDOs on Euclidean spaces}\label{subsec-kernel-criteria-euclidean}
Along with the notion of symbol functions $a(x,\xi)$ as in \eqref{eqn-symbol-bound-Rn}, pseudodifferential operators can also be characterized in terms of their \emph{Schwartz kernels}. A brief account of this characterization is given here, and for more details, see \cite[Chapter VI.4]{Stein93-Real-Variable-Methods}.

Thanks to the Schwartz kernel theorem \cite[Theorem 5.2.1]{Hormander-I}, any continuous operator $T:\csmooth(\R^n)\rightarrow \mathcal{D}'(\R^n)$ admits a Schwartz kernel $\mathcal{K}\in \mathcal{D}'(\R^n\times \R^n)$ such that $(Tf)(x)=\int_{\R^n}\mathcal{K}(x,y)\,f(y)\, dy$ (in the sense of distributions). The Schwartz kernel $\mathcal{K}_a$, as an element of $\schwartz'(\R^n\times \R^n)$, of a pseudodifferential operator with symbol $a(x,\xi)$ is given by \begin{equation}
    \mathcal{K}_a(x,y)=\int_{\R^n} a(x,\xi)\,e^{i(x-y)\cdot \xi}\,d\xi\,,
\end{equation}
and $k_a(x,z)\coloneqq K_a(x,x-z)$ \footnote{This shows that the kernel of a Fourier multiplier is just a distribution in $z$.} is the distributional inverse Fourier transform of the symbol $a(x,\cdot)$. The following theorem provides a concrete description of the distributional kernels $k_a(x,z)$.

\begin{theorem}\label{thm-kernel-estimates-Euclidean}\cite[Chapter VI.4, Proposition 1]{Stein93-Real-Variable-Methods}
    Let $a(x,\xi)$ be a symbol function of order $r$. Then the kernel $k_a(x,z)$ is in $\smooth(\R^n\times (\R^n\setminus\{0\}))$, and there exists a constant $C_{\alpha,\beta,N}$ such that \begin{equation}\label{eqn-kernel-estimate-Euclidean}
    |\partial_x^\beta \partial_z^\alpha k_a(x,z)|\leq C_{\alpha,\beta,N} |z|^{-n-r-|\alpha|- N}, \quad \text{ for all }z\neq 0, 
    \end{equation} for all multi-indices $\alpha,\beta$ and for all $N\geq 0$ such that $n+r+|\alpha|+N>0$.
\end{theorem}
In simple words, it says that $k_a(x,\cdot)$ grows with a definitive rate as $|z|\to 0$, and it decays rapidly as $|z|\to \infty$ i.e., $z\mapsto k_a(x,z)$ is a Schwartz function away from the origin. The latter property is a special feature of the Euclidean setting. We will refer to it as the Schwartz kernel $\mathcal{K}_a$ being \emph{Schwartz away from the diagonal}. It will be clear in Section \ref{subsec-Hormander-PSDO} that Hörmander pseudodifferential operators on smooth manifolds fail to have this property.

The proof of Theorem \ref{thm-kernel-estimates-Euclidean} uses a now standard technique of decomposing the symbol $a(x,\xi)$ in the $\xi$-variable, into pieces $a_j(x,\xi)$ that are supported in the dyadic annulus $\{2^{j-1}\leq |\xi|\leq 2^{j+1}\}$ for all $j\in \N$. It then reduces to estimating the kernel contributions from each of them separately, which is done using integration by parts. The statement and proof of Theorem \ref{thm-kernel-estimates-Euclidean} have inspired the main technical result of this paper, which is Theorem \ref{thm-kernel-estimates-G/K}. Moreover, the kernel estimates in Theorem \ref{thm-kernel-estimates-Euclidean} play a bigger role in the theory of pseudodifferential operators. They (almost) characterize pseudodifferential operators on $\R^n$.

\begin{theorem}[{\cite[Remark, Pg. 245]{Stein93-Real-Variable-Methods}}]\label{thm-kernel-PSDO-equivalence-Euclidean}
    Let $r<0$, and suppose $k\in \smooth(\R^n\times (\R^n\setminus\{0\}))$ is locally integrable and satisfies the estimates in \eqref{eqn-kernel-estimate-Euclidean}. Then $a(x,\xi)\coloneqq \int_{\R^n}\, k(x,z)e^{-i z\cdot \xi}\,dz$ defines a symbol function of order $r$.
\end{theorem}
\begin{remark}
    The obstruction to obtaining the perfect theorem (i.e., for all $r$) is rather simple. For $r\geq 0$, the function $z\mapsto |z|^{-n-r}$ is not locally integrable, and hence, isn't a tempered distribution on $\R^n\times \R^n$. So, one needs additional ``cancellation conditions'' on the kernel $k$ to ensure that $a(x,\xi)$ defined as in Theorem \ref{thm-kernel-PSDO-equivalence-Euclidean} is actually a symbol function of order $r$. 
\end{remark}

\subsection{Hörmander PSDOs on smooth manifolds}\label{subsec-Hormander-PSDO}
For extending the theory to smooth manifolds, one starts with the fact that the class of Euclidean pseudodifferential operators of order $r$ as defined in \eqref{eqn-PSDO-Rn} is invariant under diffeomorphisms \cite[Theorem 18.1.17]{Hormander-III}. Although the operator-class is invariant, the symbol functions are not, except the \emph{principal symbol} \footnote{On $\R^n$, the order $r$ principal symbol of $a(x,\xi)$ is defined as $[a]\in S^r(\R^n)/S^{r-1}(\R^n)$.}.

Let $M$ be a smooth manifold. A continuous linear operator $T:\csmooth(M)\rightarrow \smooth(M)$ is called a \emph{Hörmander pseudodifferential operator (or, Hörmander PSDO)} of order $r$ if for every chart $(U,\psi)$ and a cut-off function $\chi\in \csmooth(U)$, the operator $\chi T \chi : \csmooth(U)\rightarrow \csmooth(U)$ is the pull-back of an Euclidean pseudodifferential operator of order $r$ on $\psi(U)$. As stated before, a major drawback of this definition is that these operators have no well-defined notion of a \emph{complete symbol function}. The closest thing that one can get is a \emph{principal symbol}, which is defined as an equivalence class of functions on $T^*M\setminus\{0\}$. 

We now turn to their Schwartz kernels. Locally, near the diagonal $\{(m,m):m\in M\}\subset M\times M$, the Schwartz kernel of $T$ behaves like Schwartz kernels of Euclidean pseudodifferential operators. Away from the diagonal, however, the only thing that can be said is that the Schwartz kernels are smooth functions in this region. This marks a key distinction from the Euclidean theory, since Theorem \ref{thm-kernel-estimates-Euclidean} prevents arbitrary smooth functions on $\R^n\times \R^n$ from being kernels of Euclidean PSDOs on $\R^n$, and requires them to be Schwartz functions away from the diagonal.

A condition that often gets imposed on the definition of Hörmander pseudodifferential operators on manifolds is that they be \emph{properly-supported}, see \cite[Definition 18.1.21]{Hormander-III}. This is primarily required to make up for the absence of stronger off-diagonal conditions on their Schwartz kernels.

\section{Equivariant Hörmander PSDOs on $G/K$}\label{sec-equivariant-psdos-GmodK}

The purpose of this section is to look at $G$-equivariant Hörmander pseudodifferential operators on $G/K$ through the lens of the Helgason-Fourier transform (as in Definition \ref{def-helgason-fourier-transform}). To facilitate this, we will introduce the notion of \emph{Harish-Chandra symbol functions} (see Definition \ref{defn-harish-chandra-symbol}). The main result of this section is Theorem \ref{thm-Hormander-implies-HC-upgrade}, which is the easier direction of the classification of $G$-equivariant Hörmander pseudodifferential operators on $G/K$ in terms of their Harish-Chandra symbol functions.

\subsection{Kernels of $G$-equivariant operators on $G/K$}

On $G/K$, a continuous linear operator $T$ on $\csmooth(G/K)$ is said to be \emph{$G$-equivariant} (or simply, \emph{equivariant}) if $T(g\cdot f)=g\cdot (Tf)$ for all $g\in G$. Of course, such an operator will have a Schwartz kernel $\mathcal{K}\in \mathcal{D}'(G/K\times G/K)$, and the equivariance condition implies that $g.\mathcal{K}=\mathcal{K}$ for all $g\in G$, where the $G$-action is the diagonal action. However, since the action of an equivariant operator $T$ on a function $f\in \csmooth(G/K)$ is completely determined by the value $(Tf)(eK)$, it naturally defines another distribution $k\in \mathcal{D}'(G/K)$ such that $Tf=k*f$ \footnote{The convolution is technically on $G$, after transporting $k$ and $f$ to objects on $G$ by imposing right $K$-invariance.}. We will call $k$ the \emph{equivariant kernel} of $T$. The following lemma summarizes the basic properties of the equivariant kernel $k$ of an equivariant Hörmander pseudodifferential operator on $G/K$.

\begin{lemma}[{\cite[Section 3.4]{DebelloHigson}}]\label{lem-equivariant-kernel}
    Let $T:\csmooth(G/K)\rightarrow\smooth(G/K)$ be an equivariant Hörmander pseudodifferential operator. Then the equivariant kernel is a bi-$K$-invariant distribution on $G$ i.e., $k\in \mathcal{D}'(\GmodmodK)$. Furthermore, as an element of $\mathcal{D}'(G/K)$, the singular support of $k$ is contained in the identity coset $eK$. Hence, the operator $T$ can be written as 
    \begin{equation}\label{eqn-convolution-kernel-operator-GmodK}
    (Tf)(x)=\int_{G}k(y^{-1}x)f(y)\,dy=\int_{G/K} k(y^{-1}x)\,f(y)\,dy\,.
    \end{equation}
\end{lemma}

Thanks to the $KAK$ decomposition of $G$, the equivariant kernel $k$ of an equivariant Hörmander PSDO can also be seen as an element of $\mathcal{D}'(\ab)^W$ via restriction onto $\ab$. 

\subsection{Equivariant Hörmander PSDOs w.r.t. the Helgason-Fourier transform}

The definition of a $G$-equivariant Hörmander pseudodifferential operator on $G/K$ (as in Section \ref{subsec-Hormander-PSDO}) is not the most ideal for our purpose, since it is built using the Fourier transform on $\R^n$ rather than the natural harmonic analysis on $G/K$. So, the next step is to view these operators through the representation theory of $G$, and in particular, via the Helgason-Fourier transform in Definition \ref{def-helgason-fourier-transform}. A step in this direction is the following example, which shows that convolution operators (as in \eqref{eqn-convolution-kernel-operator-GmodK}) with equivariant kernels in $\HCschwartz(\GmodmodK)$ are multiplier operators with respect to the Helgason-Fourier transform whose corresponding multiplier functions belong to the Euclidean Schwartz space $\schwartz(\ab^*)$.

\begin{example}[Convolutions with $\HCschwartz(\GmodmodK)$]\label{example-convolution-schwartz}
    If $k\in \HCschwartz(\GmodmodK)$, and $T$ is the corresponding convolution operator as given in \eqref{eqn-convolution-kernel-operator-GmodK}, the Helgason-Fourier transform (as in Definition \ref{def-helgason-fourier-transform}) of $Tf$ for $f\in \HCschwartz(G/K)$ is given by $$(\widetilde{Tf})(\lambda,k)=(\mathcal{H}k)(\lambda)\cdot\widetilde{f}(\lambda,k)$$ where $\mathcal{H}k\in \schwartz(\ab^*)$ (due to Theorem \ref{thm-Schwartz-algebra-characterization}) is the spherical transform of $k$ as in \eqref{eqn-spherical-transform-defn}. It follows from \cite[Chapter III, Theorem 1.10]{HelgasonGASS} that $\widetilde{Tf}\in \schwartz(\ab^*\times K/M)$ \footnote{The Schwartz space on $\ab^*\times K/M$ is the natural extension of the Schwartz space on $\ab^*$. See \cite[Chapter III, \S 1.2]{HelgasonGASS} for the exact definition and related results.}. In particular, $T$ is a \emph{multiplier operator} with respect to the Helgason-Fourier transform, with multiplier function $(\mathcal{H}k)(\lambda)$, whose expression is given by
\begin{align}\label{eqn-multiplier-operator-GmodK}
    \begin{split}
    (Tf)(x)&=\frac{1}{|W|}\int_{\ab^*}\int_{K/M}\,\mathcal{H}k(\lambda)\,\widetilde{f}(\lambda,k)\,e_{-\lambda,k}(x)\,dk\,|\cfunc(\lambda)|^{-2}\,d\lambda\\
    &=\frac{1}{|W|}\int_{\ab^*}\int_{G/K} \HS k(\lambda) \,\varphi_{-\lambda}(y^{-1}x)\,f(y)\,dy\,|\cfunc(\lambda)|^{-2}\,d\lambda\,.
    \end{split}
\end{align}

\end{example}

The rest of this subsection will be devoted to showing that $G$-equivariant Hörmander pseudodifferential operators on $G/K$ that satisfy a \emph{rapid off-diagonal decay} condition are multiplier operators with respect to the Helgason-Fourier transform, with the multipliers being \emph{Harish-Chandra symbol functions}. We define these next, and we prove the desired result first for properly-supported, equivariant Hörmander PSDOs on $G/K$ (Theorem \ref{thm-Hormander-implies-HC}). Then, we shall upgrade this in Theorem \ref{thm-Hormander-implies-HC-upgrade}, which is the main result of this section.

\begin{definition}[Harish-Chandra Symbol functions]\label{defn-harish-chandra-symbol}
    A function $m\in \smooth(\ab^*)^W$ is called a \emph{Harish-Chandra symbol function} of order $r$ if for all multi-indices $\alpha$, there exists a constant $C_\alpha>0$ such that $$ |\partial_\lambda^\alpha m(\lambda)|\leq C_\alpha\,\langle \lambda\rangle ^{r-|\alpha|}\,.$$ The set of all Harish-Chandra symbol functions of order $r$ will be denoted by $S^{r,G}_{\operatorname{HC}}(G/K)$ \footnote{This is a Fréchet space with the standard symbol semi-norms.}.
\end{definition}

\begin{theorem}\label{thm-Hormander-implies-HC}
    Let $T:\csmooth(G/K)\rightarrow \csmooth(G/K)$ be a properly-supported, $G$-equivariant, Hörmander pseudodifferential operator of order $r$. Then $T$ is a multiplier operator with respect to the Helgason-Fourier transform with the multiplier being a Harish-Chandra symbol function $m(\lambda)\in S^{r,G}_{\operatorname{HC}}(G/K)$.
\end{theorem}

The proof of Theorem \ref{thm-Hormander-implies-HC} will take up most of the rest of this section, and we will start with the case when the \emph{order $r<0$}. An important integral formula that will be required often in what follows is the one related to the $KAK$ decomposition: \begin{equation}\label{eqn-jacobian-KAK}
    \int_G f(g)\,dg=\int_K\int_{\ab^+}\int_K f(k_1 e^H k_2)\,\delta(H)\,dk_1\,dH\,dk_2
\end{equation} where \begin{equation}\label{eqn-Jacobian}
    \delta(H)=\prod_{\alpha\in \Sigma^+}|\sinh(\alpha(H))|^{\dim \g_\alpha}
\end{equation} is the Jacobian (\cite[Theorem 5.28]{Knapp-example}).

\begin{lemma}\label{lem-Hormander-implies-HC-r-less-0}
    Let $T:\csmooth(G/K)\rightarrow \csmooth(G/K)$ be a properly-supported, $G$-equivariant, Hörmander pseudodifferential operator of order $r<0$. Then $T$ is a Helgason-Fourier multiplier operator with a smooth multiplier function $m(\lambda)\in \smooth(\ab^*)^W$.
\end{lemma}

\begin{proof}
    Since $T$ is a properly-supported, equivariant Hörmander PSDO of order $r<0$, it extends to a bounded operator on $L^2(G/K)$ (\cite[Proposition 3.4.4]{DebelloHigson}). An application of Schur's lemma (see \cite[Corollary 8.13]{Knapp-example}) (after recalling that $\pi_\lambda$'s are unitary irreducible representations of $G$) implies that $T$ is a multiplier operator, and viewing it through the Helgason-Fourier transform, the operator $T$ is given by the formula \begin{equation}\label{eqn-expression-multiplier-GmodK}
        Tf(x)=\frac{1}{|W|}\int_{\ab^*}\int_{K/M} m(\lambda)\,\widetilde{f}(\lambda,k)\,e_{-\lambda,k}(x)\,dk\,|\cfunc(\lambda)|^{-2}\,d\lambda
    \end{equation} for some $L^\infty$-function $m(\lambda)$ on $\ab^*$ that is $W$-invariant.

    The proper-supportedness of $T$ implies that its equivariant kernel $k$ is an element of $\E'(\GmodmodK)$, and hence, the expression \begin{equation}\label{eqn-expression-T-on-spherical-function}
        (T\varphi_{-\lambda})(e)=\int_{G/K} k(y^{-1})\,\varphi_{-\lambda}(y)\,dy=\int_{G/K}\, k(z)\,\varphi_{\lambda}(z)\,dz
    \end{equation} makes sense. Here, we use the symmetry property $\varphi_{-\lambda}(x^{-1})=\varphi_\lambda(x)$ in \eqref{eqn-spherical-function-symmetry}. However, since $T$ is a multiplier operator with the multiplier function $m(\lambda)$, a calculation using the inverse spherical transform in \eqref{eqn-spherical-transform-inversion} shows that the left-hand side in \eqref{eqn-expression-T-on-spherical-function} is equal to $m(\lambda)$. That is, \begin{align}\label{eqn-multiplier-expression-in-terms-kernel}\begin{split}
     m(\lambda)=(T\varphi_{-\lambda})(e)&\;\:=\;\,\int_{G/K}\,k(z) \varphi_{\lambda}(z)\,dz\\&\overset{KAK}{=}\frac{1}{|W|}\int_{\ab}\,k(e^H)\,\varphi_{\lambda}(e^{H})\,\delta(H)\,dH\,.
    \end{split}
    \end{align}
    Since we can think of the equivariant kernel as $k\in \E'(\ab)^W$, we will \emph{assume} without loss of generality that $\supp k\subset \{|H|<1\}$.
    It follows from the definition of equivariant Hörmander PSDOs on $G/K$, the proper-supportedness of $T$, the estimates in Theorem \ref{thm-kernel-estimates-Euclidean}, and the smoothness of $k$  away from the base point $eK$ (Lemma \ref{lem-equivariant-kernel}) that $|k(e^H)|\lesssim_N |H|^{-n-r-N}$ where $n=\dim G/K$ and for any $N\geq 0$ such that $n+r+N>0$. Furthermore, the Jacobian $\delta(H)$ in \eqref{eqn-Jacobian} of the $KAK$-decomposition satisfies the estimate $\delta(H)\lesssim |H|^{n-a}$ for $|H|< 1$, where $a=\dim \ab$. So, the absolute value of the integrand in the right-hand side of \eqref{eqn-multiplier-expression-in-terms-kernel} is dominated by \begin{equation}
    \varphi_0(e^H)\cdot|H|^{-n-r-N}\cdot |H|^{n-a}=\varphi_0(e^H)\cdot |H|^{-a-r-N}\overset{(\because \,|H|<1)}{\lesssim} |H|^{-a-r-N}\,.
    \end{equation} For $r<0$, we can always choose $N\geq 0$ such that $a+r+N<a$ and $n+r+N>0$, since $N$ needs to satisfy $-n-r<N<-r$. This shows that the integral in \eqref{eqn-multiplier-expression-in-terms-kernel} converges absolutely.
    
    Now, for smoothness of $m(\lambda)$, one can differentiate under the integral sign in \eqref{eqn-multiplier-expression-in-terms-kernel} i.e.,
    \begin{equation}\label{eqn-multiplier-expression-in-terms-kernel-derivative}
        \partial_\lambda^\alpha m(\lambda)=\frac{1}{|W|}\int_{\ab}\,k(e^H)\,\partial_\lambda^\alpha\varphi_{\lambda}(e^{H})\,\delta(H)\,dH\,.
    \end{equation} Now, Lemma \ref{lem-spherical-functions}$.(\emph{iii})$ says that $|\partial_\lambda^\alpha \varphi_\lambda(e^H)|\lesssim_\alpha|H|^{|\alpha|} \cdot\varphi_0(e^H)\lesssim_\alpha |H|^{|\alpha|}$ for $|H|<1$, and hence, the integral in \eqref{eqn-multiplier-expression-in-terms-kernel-derivative} converges absolutely. This shows that $m(\lambda)$ is a smooth $W$-invariant function on $\ab^*$.
\end{proof}  
Reaching the same conclusion for $r\geq 0$ presents a similar challenge to the one in the Euclidean case (see paragraph after Theorem \ref{thm-kernel-PSDO-equivalence-Euclidean}). However, the following lemma, which is a folklore trick using the Laplace-Beltrami operator, can be used to get around it.

\begin{lemma}[Laplacian Trick]\label{lem-Laplacian-trick}
    Let $T$ be a properly-supported, Hörmander PSDO of order $r$. For all $N\in \N$, there exists a properly-supported, Hörmander PSDO $T'$ of order $r-2N$ such that $T-\Delta^N\circ T'$ is a properly-supported smoothing \footnote{A smoothing operator on a smooth manifold $M$ will mean an integral operator with a smooth kernel $\mathcal{K}\in \smooth(M\times M)$.} operator where $\Delta$ is the Laplace-Beltrami operator.
\end{lemma}

\begin{proof}
    Let $N\in \N$. Since $\Delta^N$ is an elliptic differential operator, it admits a parametrix $Q_N$ that can be chosen to be a properly-supported, Hörmander PSDO of order $-2N$ such that $$\Delta^N\circ Q_N + R_N =I$$ where $R_N$ is a properly-supported smoothing operator for all $N\geq 0$ (\cite[Theorem 18.1.24]{Hormander-III}). So, $T$ can be written as \begin{equation}\label{eqn-T-in-terms-Laplacian}
        T= \Delta^N \circ (Q_N \circ T) + R_N\circ T\,.
    \end{equation} The operator $T'\coloneq Q_N\circ T$ satisfies the desired conclusions of the lemma.
\end{proof}

Now, returning to a properly-supported, $G$-equivariant Hörmander PSDO $T$ of order $r\geq 0$, we choose $N$ large enough such that $r-2N<0$. Then, by Lemma \ref{lem-Laplacian-trick}, there exists a properly-supported, $G$-equivariant, Hörmander PSDO $T'$ of order $r-2N<0$ such that $R'\coloneq T-\Delta^N\circ T'$ is a properly-supported, $G$-equivariant smoothing operator. By Lemma \ref{lem-Hormander-implies-HC-r-less-0}, $T'$ and $R'$ are Helgason-Fourier multiplier operators with smooth multiplier functions, and it follows from Lemma \ref{lem-spherical-functions}.$(\emph{i})$ that $\Delta^N$ is a Helgason-Fourier multiplier operator with the multiplier function $(-1)^N(|\lambda|^2+|\rho|^2)^N$. This proves that any properly-supported, equivariant Hörmander PSDO $T$ of order $r$ is a Helgason-Fourier multiplier operator with a smooth multiplier function, which is the first part of Theorem \ref{thm-Hormander-implies-HC}. 

Before we proceed to the proof of the second part of Theorem \ref{thm-Hormander-implies-HC}, we need to introduce a construction and a lemma. 
\begin{construction}\label{chopping-0}
        We decompose the punctured ball $\{z\in G/K\,:0<|z|\leq 1\}$ into dyadic annuli $\{2^{-j-1}\leq |z|\leq 2^{-j+1}\}$ for all $j\in \N$. There exists a smooth left-$K$-invariant partition of unity $\{\eta_j:G/K\rightarrow \R\}_{j\in \N}$ of the punctured ball such that $\supp \eta_j\subset \{2^{-j-1}\leq |z|\leq 2^{-j+1}\}$ and $\sum_{j=1}^\infty \eta_j(z)=1$ for all $z\neq e$. Using this partition of unity, let us define the functions
        \begin{equation}
            k_j(z)\coloneqq \eta_j(z)\cdot k(z)\quad \text{ and }\quad m_j(\lambda)\coloneqq(\HS k_j)(\lambda) \quad \text{for all }j\in \N.
        \end{equation}  

\end{construction}
\begin{lemma}\label{lem-estimate-m-j}
        For all $M\in \mathbb{N}\cup\{0\}$ such that $n+r+M>0$, there exists a constant $C_M>0$ such that $$|m_j(\lambda)|\leq C_M\,\langle \lambda\rangle^{-M}\,2^{j(r+M)}\,.$$
    \end{lemma}
    \begin{proof}
        Remember that the Laplace-Beltrami operator $\Delta$ on $G/K$ is essentially self-adjoint, and satisfies the eigenvalue equations $\Delta\varphi_\lambda=-(|\lambda|^2+|\rho|^2)\,\varphi_\lambda$ (Lemma \ref{lem-spherical-functions}.$(\emph{i})$). We will denote $p_\Delta(\lambda)\coloneqq-(|\lambda|^2+|\rho|^2)$, and write \begin{align}\label{eqn-estimate-m-j-lemma-1}
            \begin{split}
            \left(p_\Delta(\lambda)\right)^N\,m_j(\lambda)&=\int_{G/K} k_j(z)\,\Delta^N \varphi_\lambda(z)\,dz=\langle k_j,\Delta^N \varphi_{-\lambda}\rangle=\langle \Delta^N k_j,\varphi_{-\lambda}\rangle\,.
            \end{split}
        \end{align}
        If we choose $N\in \mathbb{N}\cup\{0\}$ such that $n+r+2N>0$, then the estimates in Theorem \ref{thm-kernel-estimates-Euclidean} and the definition of equivariant Hörmander PSDOs on $G/K$ using local charts imply that $|\Delta^Nk(z)|\lesssim_N |z|^{-n-r-2N}$ for $|z|<1$.  Hence, \begin{align}\label{eqn-estimate-m-j-lemma}
            \begin{split}
            |\left(p_\Delta(\lambda)\right)^N\,m_j(\lambda)|&\;\:\lesssim_N\int_{2^{-j-1}\leq |z|\leq 2^{-j+1}} |\Delta^Nk_j(z)|\,dz\\&\overset{KAK}{=}\int_{2^{-j-1}\leq |H|\leq 2^{-j+1}} |\Delta^Nk_j(H)|\,\delta(H)\,dH\lesssim_N 2^{j(r+2N)}\,.
            \end{split}
        \end{align}Here, we use the fact that the Jacobian in \eqref{eqn-Jacobian} satisfies $\delta(H)\lesssim |H|^{n-a}$ for $|H|<1$. This completes the proof for non-negative even integers $M=2N$. 
        
        For odd integers $M$, it is enough to show it for $M=1$ i.e., assume that $n+r+1>0$. We will use Green's formula for Riemannian manifolds on $G/K$ (see the proof of \cite[Lemma 16]{Anker90}) for this analysis. With $N=1$ in \eqref{eqn-estimate-m-j-lemma-1}, we can rewrite it as  \begin{equation}\label{eqn-Green-formula}
            p_\Delta(\lambda)m_j(\lambda)=\langle k_j,\Delta \varphi_{-\lambda}\rangle\overset{\text{Green's formula}}{=}-\langle\nabla k_j, \nabla \varphi_{-\lambda}\rangle
        \end{equation} where $\nabla$ is the gradient vector field on $G/K$. Furthermore, Lemma \ref{lem-spherical-functions}.(\emph{ii}) says that $|\nabla \varphi_{-\lambda}(z)|\lesssim \langle \lambda\rangle$, and the kernel estimates in Theorem \ref{thm-kernel-estimates-Euclidean} imply that $|\nabla k(z)|\lesssim |z|^{-n-r-1}$. An identical analysis to \eqref{eqn-estimate-m-j-lemma} using \eqref{eqn-Green-formula} shows that $$|m_j(\lambda)|\lesssim \langle\lambda\rangle^{-1}\cdot 2^{j(r+1)}\,.$$ So, using interpolation, the lemma follows for odd integers $M$ whenever $n+r+M>0$, and hence, for all $M\in \mathbb{N}\cup\{0\}$ such that $n+r+M>0$.
    \end{proof}
    Both Construction \ref{chopping-0} and Lemma \ref{lem-estimate-m-j} are inspired by the proof of Theorem \ref{thm-kernel-estimates-Euclidean} in \cite[Chapter VI.4, Proposition 1]{Stein93-Real-Variable-Methods}.

\begin{proof}[Proof of Theorem \ref{thm-Hormander-implies-HC}]
    All that remains to be proved are the estimates in Definition \ref{defn-harish-chandra-symbol} for the smooth multiplier function $m(\lambda)$. However, thanks to the Laplacian trick in Lemma \ref{lem-Laplacian-trick}, it is enough to prove the estimates for the multiplier functions $m(\lambda)$ corresponding to equivariant Hörmander PSDOs of order $r<0$.

    So, \emph{fix} $r<0$. Since $T$ is properly-supported, it can be assumed that $\supp k\subset \{|z|<1\}$, and we need to check the estimates in Definition \ref{defn-harish-chandra-symbol} for $\langle \lambda\rangle\geq 1$. So, the expression in \eqref{eqn-multiplier-expression-in-terms-kernel} becomes \begin{equation}\label{eqn-breaking-the-kernel}
        m(\lambda)=\underbrace{\int_{|z|\leq \langle \lambda\rangle^{-1}} k(z)\varphi_\lambda(z)\,dz}_{\Circled{I}}+\underbrace{\int_{\langle \lambda\rangle^{-1}<|z|\leq 1} k(z)\varphi_\lambda(z)\,dz}_{\Circled{II}}
    \end{equation} after decomposing the integral over $G/K$ into two parts using the distance from $eK$ function $|\cdot|_{G/K}$ on $G/K$. Following Construction \ref{chopping-0}, we can estimate $|\Circled{I}|$ by the (infinite) sum $\sum_{2^j>\langle \lambda\rangle}|m_j(\lambda)|$; and estimate $|\Circled{II}|$ by the (finite) sum $\sum_{2^j<\langle \lambda\rangle}|m_j(\lambda)|$. We will estimate these two sums separately.
    \begin{itemize}
        \item For estimating $\Circled{I}$, we will \emph{fix} $M\in \mathbb{N}\cup\{0\}$ such that $n+r+M>0$ and $r+M<0$. This is always possible since $r<0$, and we need $M$ to satisfy $-n-r<M<-r$ \footnote{Note that $\operatorname{dim}(G/K)=n\geq 2$.}. Then \begin{align}\label{eqn-I-preliminary}
            \begin{split}
            |\Circled{I}|\lesssim \sum_{2^j>\langle \lambda\rangle} |m_j(\lambda)|&\overset{\text{Lemma }\eqref{lem-estimate-m-j}}{\lesssim} \sum_{2^j>\langle \lambda\rangle}\langle \lambda\rangle^{-M} \,2^{j(r+M)}\\&\;\;\;
            \overset{(r+M<0)}{\lesssim}\langle \lambda\rangle^{-M}\langle \lambda\rangle^{r+M}=\langle \lambda\rangle^{r},
            \end{split}
        \end{align}
        where the last inequality follows from estimating the sum using the (infinite) geometric series formula.
        \item For estimating $\Circled{II}$, we will \emph{fix} $M\geq 0$ such that $r+M>0$. Then \begin{align}\label{eqn-II-preliminary}
            \begin{split}
            |\Circled{II}|\lesssim \sum_{2^j<\langle \lambda\rangle} |m_j(\lambda)|&\overset{\text{Lemma }\eqref{lem-estimate-m-j}}{\lesssim} \sum_{2^j<\langle \lambda\rangle}\langle \lambda\rangle^{-M} \,2^{j(r+M)}\\&\;\;\;
            \overset{(r+M>0)}{\lesssim}\langle \lambda\rangle^{-M}\langle \lambda\rangle^{r+M}=\langle \lambda\rangle^{r}
            \end{split}
        \end{align} by estimating the sum using the (finite) geometric series formula.
    \end{itemize}
    So, \eqref{eqn-I-preliminary} and \eqref{eqn-II-preliminary} applied to \eqref{eqn-breaking-the-kernel} complete the estimate for $m(\lambda)$. The estimates in Definition \ref{defn-harish-chandra-symbol} for the $\lambda$-derivatives of $m(\lambda)$ can be proved in a similar fashion to the argument at the end of the proof of Lemma \ref{lem-Hormander-implies-HC-r-less-0}. This completes the proof of Theorem \ref{thm-Hormander-implies-HC}.
\end{proof}

A related result is proved in \cite{DebelloHigson} except that only the principal symbol is considered there. In contrast, Theorem \ref{thm-Hormander-implies-HC} shows that properly-supported, $G$-equivariant Hörmander pseudodifferential operators on $G/K$ have \emph{complete} symbol functions, which are Harish-Chandra symbol functions $m(\lambda)$ as in Definition \ref{defn-harish-chandra-symbol}. We will state a mild upgrade of Theorem \ref{thm-Hormander-implies-HC} up next. If we add the assumption of the $G/K$-version of ``rapid-off-diagonal-decay" (see Definition \ref{defn-rapid-off-diagonal-decay} for the precise definition) to the equivariant kernel $k$ of an equivariant Hörmander PSDO, similar to the one in Theorem \ref{thm-kernel-estimates-Euclidean}, then we can drop the proper-supportedness condition on $T$ in Theorem \ref{thm-Hormander-implies-HC}, and still conclude that $T$ is a multiplier operator with a Harish-Chandra symbol function of order $r$. The naturality and the necessity of this extra hypothesis of rapid off-diagonal decay in Theorem \ref{thm-Hormander-implies-HC-upgrade} will become clear in Theorem \ref{thm-HC-implies-hormander}.

\begin{definition}\label{defn-rapid-off-diagonal-decay}
    A distribution $u\in \mathcal{D}'(\GmodmodK)$ is in $\HCschwartz(\GmodmodK)$ \emph{away from the diagonal} if there exists $\chi\in C_c^\infty(\GmodmodK)$ with uniformly bounded derivatives such that $(1-\chi)\cdot u\in \HCschwartz(\GmodmodK)$.
\end{definition}

\begin{theorem}\label{thm-Hormander-implies-HC-upgrade}
    Let $T:\csmooth(G/K)\rightarrow \smooth(G/K)$ be an equivariant, Hörmander pseudodifferential operator of order $r$ whose equivariant kernel $k$ (as in \eqref{eqn-convolution-kernel-operator-GmodK}) is an element of $\HCschwartz(\GmodmodK)$ away from the diagonal. Then $T$ is a multiplier operator with respect to the Helgason-Fourier transform, with the multiplier $m(\lambda)\in S^{r,G}_{\operatorname{HC}}(G/K)$.
\end{theorem}

\begin{proof}
    Let $k=k_0+k_\infty$ such that $k_0$ is supported in $\{|z|\leq 2\}$ and $k_\infty$ is supported in $\{|z|\geq 1\}$. Now, the operator $T$ can be written as $T=T_0+T_\infty$, where $T_0$ and $T_\infty$ are convolution operators with kernels $k_0$ and $k_\infty$ respectively. Naturally, $T_0$ is a properly-supported, $G$-equivariant, Hörmander pseudodifferential operator of order $r$, and hence, a multiplier operator with a Harish-Chandra symbol function of order $r$ due to Theorem \ref{thm-Hormander-implies-HC}. On the other hand, thanks to the discussion in Section \ref{subsec-Hormander-PSDO}, $k_\infty\in \smooth(\GmodmodK)$, and in particular, $k_\infty\in \HCschwartz(\GmodmodK)$. It follows from Example \ref{example-convolution-schwartz} that $T_\infty$ is a multiplier operator with a Harish-Chandra symbol function in $\schwartz(\ab^*)^W$. So, $T=T_0+T_\infty$ is a multiplier operator with respect to the Helgason-Fourier transform, with a Harish-Chandra symbol function of order $r$.
\end{proof}

\section{Equivariant Harish-Chandra PSDOs on $G/K$}

This brings us to the primary goal of the paper --- to prove the exact converse to Theorem \ref{thm-Hormander-implies-HC-upgrade} i.e., to show that multiplier operators with respect to the Helgason-Fourier transform associated to Harish-Chandra symbol functions are equivariant Hörmander PSDOs on $G/K$ with their equivariant kernel $k$ (as in \eqref{eqn-convolution-kernel-operator-GmodK}) in $\HCschwartz(\GmodmodK)$ away from the diagonal. To do so, we will introduce the notion of equivariant Harish-Chandra pseudodifferential operators on $G/K$, and then prove the main technical result of this paper that deals with kernel estimates of equivariant Harish-Chandra PSDOs (Theorem \ref{thm-kernel-estimates-G/K}). 

\subsection{Definitions and main results}

\begin{definition}[Equivariant Harish-Chandra PSDOs]\label{defn-harish-chandra-PSDO}
    A $G$-equivariant operator $T_m:\csmooth(G/K)\to \smooth(G/K)$ will be called an \emph{equivariant Harish-Chandra pseudodifferential operator} (or simply, an \emph{equivariant Harish-Chandra PSDO}) of order $r$ if it is a multiplier operator with a Harish-Chandra symbol function $m(\lambda)$ of order $r$ with respect to the Helgason-Fourier transform i.e., $T_m$ is given by the expression in \eqref{eqn-expression-multiplier-GmodK}. We shall denote by $\Psi^{r,G}_{\text{HC}}(G/K)$ the set of all equivariant Harish-Chandra PSDOs of order $r$ \footnote{The Fréchet space structure on $\Psi^{r,G}_{\text{HC}}(G/K)$ comes from the Fréchet space $S^{r,G}_{\text{HC}}(G/K)$.}.
\end{definition}

As discussed in Section \ref{sec-equivariant-psdos-GmodK}, every continuous equivariant operator on $G/K$ has an equivariant kernel $k$ (Lemma \ref{lem-equivariant-kernel}) such that $(Tf)(x)=\int_{G/K}k(y^{-1}x)\,f(y)\,dy$. In the case of an equivariant Harish-Chandra PSDO $T_m$, the equivariant kernel $k$ is given by the inverse spherical transform \eqref{eqn-spherical-transform-inversion} of $m$ i.e., \begin{equation}
    k=\HS^{-1}m.
\end{equation} The estimates in Definition \ref{defn-harish-chandra-symbol} ensure that $m\in \schwartz'(\ab^*)^W$, and hence, $k=\HS^{-1}m\in \HCschwartz'(\GmodmodK)$ (as in Section \ref{subsec-schwartz-space}).
The following theorem, which is the main technical result on which the paper relies, can now be stated.

\begin{theorem}\label{thm-kernel-estimates-G/K}
    Let $m\in S^{r,G}_{\operatorname{HC}}(G/K)$, and let $T_m$ be the corresponding equivariant Harish-Chandra PSDO in $\Psi^{r,G}_{\operatorname{HC}}(G/K)$. The equivariant kernel $k(z)$ of the operator $T_m$ is smooth for $z\neq e$. Moreover, for all $N\geq 0$ such that $n+r+d_1+d_2+N>0$, \begin{equation}
    |k(D_1;z;D_2)|\lesssim_{N,D_1,D_2}\, \varphi_0(z)\cdot |z|^{-n-r-d_1-d_2-N} \tag{$*_{r}$}
    \end{equation} for $z\neq 0$ where $D_1,D_2\in U(\g)$ and $d_i=\deg D_i$ for $i=1,2$, and $n=\dim G/K$.
\end{theorem}

We will give a proof of this theorem in Section \ref{sec-proof-kernel-estimates-G/K}. For the remainder of this section, we will look at its implications. To get things started, Theorem \ref{thm-kernel-estimates-G/K} implies the converse to Theorem \ref{thm-Hormander-implies-HC-upgrade}, which is the following theorem.

\begin{theorem}\label{thm-HC-implies-hormander}
    Every operator $T\in \Psi^{r,G}_{\operatorname{HC}}(G/K)$ is a $G$-equivariant Hörmander pseudodifferential operator on $G/K$, as per the definition in Section \ref{subsec-Hormander-PSDO}, with its equivariant kernel $k$ belonging to $\HCschwartz(\GmodmodK)$ away from the diagonal.
\end{theorem}

\begin{proof}
    Firstly, Theorem \ref{thm-kernel-estimates-G/K} implies that the equivariant kernel $k$ of $T$ is an element of $\HCschwartz(\GmodmodK)$ away from the diagonal.
    
    To prove the rest of the theorem, let us first assume that $m\in S^{r,G}_{\operatorname{HC}}(G/K)$ for some $r<0$. From the definition of a Hörmander PSDO using local charts, and Theorem \ref{thm-kernel-PSDO-equivalence-Euclidean}, an operator $(T_m f)(x)=\int_{G/K}\mathcal{K}(x,y)\,f(y)\,dy$ is a Hörmander PSDO on $G/K$ if the Schwartz kernel $\mathcal{K}(x,y)$ satisfies the decay estimates of Theorem \ref{thm-kernel-estimates-Euclidean} in a coordinate patch near the diagonal. These estimates follow readily from Theorem \ref{thm-kernel-estimates-G/K}. 

    Suppose now that $m\in S^{r,G}_{\operatorname{HC}}(G/K)$ for some $r\geq 0$. Then, we will again use a trick involving the Laplace-Beltrami operator, similar in spirit to Lemma \ref{lem-Laplacian-trick}. Remember the Harish-Chandra symbol of the Laplace-Beltrami operator $\Delta$ is $-(|\lambda|^2+|\rho|^2)$, which is clearly invertible. So, the operator $\Delta^{-1}$ with Harish-Chandra symbol $-(|\lambda|^2+|\rho|^2)^{-1}$ is an equivariant Harish-Chandra PSDO of order $-2$. Of course, for $N$ sufficiently large, $T_m\circ (\Delta^{-1})^N$ is an equivariant Harish-Chandra PSDO of order $r-2N<0$, and hence, an equivariant Hörmander PSDO on $G/K$. Since $\Delta$ is properly-supported, $$T_m=(T_m\circ (\Delta^{-1})^N)\circ \Delta^N$$ is an equivariant Hörmander pseudodifferential operator on $G/K$ as well.
\end{proof}

It is evident from Theorem \ref{thm-HC-implies-hormander} that the rapid off-diagonal decay property of the equivariant kernel is actually a necessary condition for a $G$-equivariant operator to have a Harish-Chandra symbol function. In particular, if \begin{equation}
    \Psi^{-\infty,G}_{\operatorname{HC}}(G/K) \coloneqq  \bigcap_{r} \,\Psi^{r,G}_{\operatorname{HC}}(G/K)
\end{equation} is the set of equivariant Harish-Chandra pseudodifferential operators of order $-\infty$, then we have the following theorem.

\begin{theorem}\label{thm-smoothing=schwartz}
    The inclusion $\HCschwartz(\GmodmodK)\hookrightarrow \Psi^{-\infty,G}_{\operatorname{HC}}(G/K)$ is an isomorphism of Fréchet algebras. Furthermore, every operator $T\in \Psi^{r,G}_{\operatorname{HC}}(G/K)$ extends to a continuous linear operator $T:\HCschwartz(G/K)\rightarrow \HCschwartz(G/K)$.
\end{theorem}

\begin{proof}
    Let $T\in \Psi^{-\infty,G}_{\operatorname{HC}}(G/K)$ and let $k$ be its equivariant kernel. Let $k=k_0+k_\infty$ as in Theorem \ref{thm-Hormander-implies-HC-upgrade}. It follows from Theorem \ref{thm-kernel-estimates-G/K} that $k_\infty$ is an element of $\HCschwartz(\GmodmodK)$. On the other hand, $k_0$ is the equivariant kernel of a properly-supported, equivariant Hörmander PSDO of order $-\infty$, and hence, is an element of $\csmooth(\GmodmodK)$. So, $k=k_0+k_\infty$ is an element of $\HCschwartz(\GmodmodK)$. The Fréchet algebra isomorphism then follows from Theorem \ref{thm-Schwartz-algebra-characterization}.

    It is clear from Example \ref{example-convolution-schwartz} that every operator $T\in \Psi^{-\infty,G}_{\operatorname{HC}}(G/K)$ extends to a continuous linear operator on $\HCschwartz(G/K)$. The result for operators of order $r\neq -\infty$ follows from a standard argument by approximating an order $r$ Harish-Chandra symbol function $m(\lambda)$ by a sequence of Schwartz functions on $\ab^*$.
\end{proof}

\begin{remark}
    Theorem \ref{thm-smoothing=schwartz} implies that the correct notion of an equivariant ``smoothing" operator on $G/K$ is an integral operator with its equivariant kernel in $\HCschwartz(\GmodmodK)$. Moreover, this class of equivariant smoothing operators forms a closed two-sided ideal in the algebra of equivariant Harish-Chandra pseudodifferential operators on $G/K$.
\end{remark}

\subsection{Proof of Theorem \ref{thm-kernel-estimates-G/K}}\label{sec-proof-kernel-estimates-G/K} The proof will be divided into 10 steps. Let us recall that we need to estimate the quantity \begin{equation}
    |k(D_1;z;D_2)|\quad \text{ for }z\neq e \text{ and }D_1,D_2\in U(\g)\,.
\end{equation}

\textbf{Step 1: (Reductions)} For $D_1,D_2\in U(\g)$, one can write, by differentiating under the integral sign, that \begin{equation}\label{eqn-kernel-expression-D1-D2}
    k(D_1;z;D_2)=\frac{1}{|W|}\int_{\ab^*} m(\lambda)\,\varphi_{-\lambda}(D_1;z;D_2)\,|\cfunc(\lambda)|^{-2}\,d\lambda\,.
\end{equation} But, from Lemma \ref{lem-spherical-functions}.$(\emph{ii})$, we have the estimate $|\varphi_{-\lambda}(D_1;z;D_2)|\lesssim \langle \lambda\rangle^{d_1+d_2}\cdot\varphi_0(z)$, which says that for estimating \eqref{eqn-kernel-expression-D1-D2}, it is enough to find the estimates for the equivariant kernel (without any differential operators) corresponding to the Harish-Chandra symbol function $\langle \lambda\rangle^{d_1+d_2}\cdot m(\lambda)$ of order $r+d_1+d_2$. Hence, it is enough to prove the estimates for $k(z)$ without any differential operators.

Furthermore, since the kernel $k$ is bi-$K$-invariant, it is enough to prove the estimates for $|k(e^{H_0})|$ for $H_0\in \overline{\ab^+}\setminus\{0\}$. In conclusion, we only need to estimate the absolute value of \begin{equation}\label{eqn-kernel-expression-h0}
    k(e^{H_0})=\frac{1}{|W|}\int_{\ab^*} m(\lambda)\,\varphi_{-\lambda}(e^{H_0})\,|\cfunc(\lambda)|^{-2}\,d\lambda
\end{equation} for all $H_0\in \overline{\ab^+}\setminus\{0\}$. 

Finally, for Steps 2 through 9, we will \emph{assume} that the order $r>-a$; the remaining case, $r\leq -a$ is handled separately in Step 10 using Lemma \ref{lem-Laplacian-trick-upgraded}. So, until Step 9, it suffices to prove the estimates \begin{equation}
\quad\quad\quad|k(e^{H_0})|\lesssim_{N}\varphi_0(H_0)\cdot|H_0|^{-n-r-N}\quad\text{ for all }N\geq 0,\; H_0\in \overline{\ab^+}\setminus\{0\},
\end{equation} for all equivariant Harish-Chandra PSDOs of order $r> -a$. \\

\textbf{Step 2: (Chopping on $\ab^*$)} To start off, using the commutative diagram in \eqref{diagram-Abel-triangle}, define \begin{equation}
\ell\coloneq \mathcal{F}^{-1}(\mathcal{I}m)=\A(k)\in \E'(\ab)^W\,.
\end{equation} We will chop up the symbol function $m(\lambda)$ into pieces supported in dyadic annuli. This resembles the construction appearing in Stein's proof of Theorem \ref{thm-kernel-estimates-Euclidean}.

\begin{construction}\label{Chopping - 1}
Let us fix a bump function $\zeta_0\in \csmooth(\ab^*)^W$ such that $|(\partial^\alpha\zeta_0)(\lambda)|\leq 1$ for all multi-indices $\alpha$, and $$\zeta_0(\lambda)=\begin{cases}
    1, \quad |\lambda|\leq 1,\text{ and }\\ 
    0, \quad |\lambda|\geq 2\,.
    \end{cases}$$ We then define $\zeta_j(\lambda)\coloneq\zeta_0(2^{-j}\lambda)-\zeta_0(2^{-j+1}\lambda)$ for all $j\in \N$. Then $\sum_{j=0}^\infty\zeta_j(\lambda)=1$ for all $\lambda\in \ab^*$, and $\operatorname{supp}(\zeta_j)\subset \{\lambda\in \ab^*\,:\,2^{j-1}\leq |\lambda|\leq 2^{j+1}\}$ for all $j\in \N$. $\{\zeta_j\}$ is an example of a \emph{dyadic partition of unity} in $\ab^*$. See \cite[Chapter VI, Section 4.1]{Stein93-Real-Variable-Methods} for more details. Now, define \begin{equation}\label{eqn-definition-j}
    m_j(\lambda)\coloneq m(\lambda)\cdot \zeta_j(\lambda)\quad\text{ and }\quad k_j\coloneq \HS^{-1}(m_j),\quad\ell_j\coloneq \F^{-1}(\mathcal{I}m_j)\,.
    \end{equation} The following diagram summarizes the process, and the starting point is $m_j$. \begin{equation}\label{eqn-diagram-j}
        \begin{tikzcd}
    m_j \arrow[d,mapsto,"\HS^{-1}"']\arrow[r, mapsto, "\mathcal{I}^{-1}"] & \F (\ell_j)\arrow[d,mapsto, "\F^{-1}"] \\
    k_j   \arrow[r,mapsto, "\A"]&  \ell_j
    \end{tikzcd}
    \end{equation} 
     \begin{remark}
    The chopping in Construction \ref{Chopping - 1} is taking place over $\ab^*$ (in principle, the spherical tempered dual). Contrast it with Construction \ref{chopping-0} in the proof of Theorem \ref{thm-Hormander-implies-HC}, which involved chopping over the punctured disk in $G/K$.
\end{remark}

Since $m_j$'s are compactly supported, $k_j$'s and $\ell_j$'s are smooth functions on $G/K$ and $\ab$ respectively. Also, from Euclidean analysis, there are straightforward bounds for the function $\ell_j$.

    \begin{lemma}[{\cite[Lemma, Pg. 244]    {Stein93-Real-Variable-Methods}}]\label{lem-l-j-bound}
        For all $M\in \mathbb{R}_{\geq 0}$, there exists a constant $C_{\alpha,M}>0$ such that $$|\left(\partial^\alpha\ell_j\right)(H)|\leq C_{\alpha,M}\cdot |H|^{-M}\cdot 2^{j(a+r-M+|\alpha|)}$$ for all $H\neq 0$ and $j\in \N$.
    \end{lemma}
    \begin{proof}
        The proof in \cite[Lemma, Pg. 244]{Stein93-Real-Variable-Methods} using an integration by parts argument is given for only $M\in \N\cup\{0\}$. The result for all $M\in \mathbb{R}_{\geq 0}$ follows from interpolation.
    \end{proof}
   
Moreover, since $\{\zeta_j\}$ in Construction \ref{Chopping - 1} is a partition of unity, it follows that (as a locally finite sum) \begin{equation}\label{eqn-mj-decomposition}
    m=\sum_{j=0}^\infty m_j\,,
    \end{equation} and each of the $m_j$'s satisfies the symbol estimates in Definition \ref{defn-harish-chandra-symbol} with the constants independent of $j$. So, to estimate $|k(e^{H_0})|$, it is enough to estimate $\sum_{j=0}^\infty|k_j(e^{H_0})|$.
\end{construction}

\textbf{Step 3: (Chopping on $\ab$)} Construction \ref{Chopping - 2}, stated below, is a minor variation of Anker's method in \cite{Anker92_Sharp_Estimates_Laplacian}.

\begin{construction}\label{Chopping - 2}
    Let us \emph{fix} $H_0\in \overline{\ab^+}$ non-zero. We will need a function $\omega^{H_0}\in \csmooth(\ab)^W$ (see Figure \ref{fig-cutoff-H-epsilon}) such that \begin{align*}
            \omega^{H_0}(H)=\begin{cases}
                1 & \text{ if }|H|\leq \frac{|H_0|}{2}\,,\\
                0 & \text{ if }|H|\geq \frac{3|H_0|}{4}\,, 
            \end{cases}
    \end{align*} and for any multi-index $\alpha\neq 0$, the functions $\partial^\alpha \omega^{H_0}(H)$ are supported in the annulus $\{\frac{|H_0|}{2} < |H| < \frac{3|H_0|}{4}\}$ and satisfy the bounds \begin{equation}\label{eqn-derivative-pointwise-bounds-cutoff-H-epsilon}
    |\partial^\alpha \omega^{H_0}(H)|\leq C_\alpha \cdot |H_0|^{-|\alpha|}
    \end{equation} for all $H\in \ab$ and for some constant $C_\alpha>0$. The function $\omega^{H_0}$ can be constructed by choosing a smooth function $\omega\in \smooth(\R)$ such that $\omega(t)=1$ for $t\leq \frac{1}{2}$ and $\omega(t)=0$ for $t\geq \frac{3}{4}$, such that $\left|\frac{d^i\omega}{dt^i}(t)\right|\leq 1$ for all $i\geq 0$, and defining \begin{equation}\label{eqn-cutoff-H-0-omega}
        \omega^{H_0}(H)\coloneq \omega\left(|H|/|H_0|\right)\,.
    \end{equation} Similarly, define $\Omega^{H_0}\in \smooth(\ab)^W$ such that $\Omega^{H_0}\coloneqq 1-\omega^{H_0}$. Clearly, the derivatives $\partial^\alpha\Omega^{H_0}$ also satisfy \eqref{eqn-derivative-pointwise-bounds-cutoff-H-epsilon} for all multi-indices $\alpha\neq 0$. 

    \begin{figure}
        \centering
        \caption{The graph of $\omega(t)$ and the support of $\omega^{H_0}(H)$ in $\ab$ as defined in \eqref{eqn-cutoff-H-0-omega}.}
        \label{fig-cutoff-H-epsilon}
        \begin{tikzpicture}[scale=.9]

  \begin{scope}[xshift=-10.5cm]
    \draw[-latex] (-0.5,0) -- (4,0) node[right] {\footnotesize $t$};
    \draw[-latex] (0,-0.5) -- (0,2) node[above] {\footnotesize $\omega(t)$};

    \draw[dashed, gray] (1.6,0) -- (1.6,1.6);
    \draw[dashed, gray] (2.4,0) -- (2.4,1.6);
    \draw[dashed, gray] (3.2,0) -- (3.2,1.6);
    \draw[dashed, gray] (0,1) -- (3.5,1);

    \node[below] at (1.6,0) {\footnotesize $\frac{1}{2}$};
    \node[below] at (2.4,0) {\footnotesize $\frac{3}{4}$};
    \node[below] at (3.2,0) {\footnotesize $1$};
    \node[left] at (0,1) {\footnotesize $1$};

    \draw[very thick]
      (0,1) -- (1.6,1)
      .. controls (1.9,1) and (2.2,0) .. (2.4,0)
      -- (3.5,0);

  \end{scope}

  \draw[-latex, thick] (-6.1,1) -- (-3.5,1) node[midway, above] {\scriptsize Extended radially to $\ab$};

  \def\Rout{3}      
  \def\Rmid{2.25}    
  \def\Rin{1.5}     

  \draw[fill=white] (0,0) circle (\Rout);

  \draw[fill=gray!30] (0,0) circle (\Rmid);

  \draw[fill=gray!70] (0,0) circle (\Rin);

  \draw (0,0) circle (\Rout);
  \draw (0,0) circle (\Rmid);
  \draw (0,0) circle (\Rin);

  \draw[-latex] (0,0) -- (30:\Rin)  node[left] {\footnotesize $\frac{|H_0|}{2}\;\;$};
  \draw[-latex] (0,0) -- (240:\Rmid) node[right]  {\footnotesize $\;\;\;\,\frac{3|H_0|}{4}$};
  \draw[-latex] (0,0) -- (-30:\Rout) node[below right] {\footnotesize $|H_0|$};

  \fill (0,0) circle (2pt);

  \node[anchor=north west] at (\Rout, \Rout+.1) {$\ab$};

\end{tikzpicture}

    \end{figure}
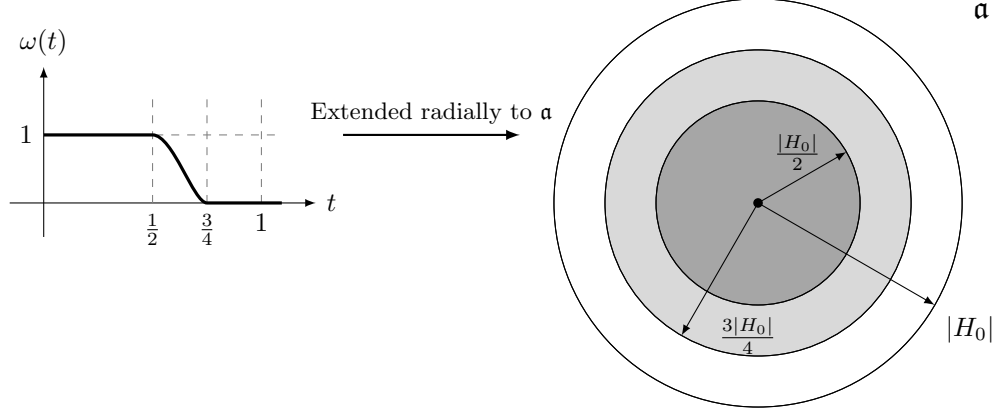

    Using the commutative diagram in \eqref{diagram-Abel-triangle}, these cut-off functions $\omega^{H_0}$ and $\Omega^{H_0}$ will define the following chopped-off versions of the $\ell_j$'s, $k_j$'s and $m_j$'s, say
    \begin{align}\label{eqn-definition-j-h-epsilon}
        \begin{split}
            \ell_j^{H_0}(H)&\coloneq \ell_j(H)\cdot \Omega^{H_0}(H)\quad \quad \text{ and }\quad \quad k_j^{H_0}\coloneq \A^{-1}(\ell_j^{H_0})\, ,\\& m_j^{H_0}\coloneq (\mathcal{I}^{-1}\circ\F)(\ell_j^{H_0})=m_j-m_j*\check{\omega}^{H_0}\, ,
        \end{split}
    \end{align}
    where $\check{\omega}^{H_0}\coloneqq(\mathcal{I}^{-1}\circ \F)(\omega^{H_0})$. The following diagram summarizes the process, and again, the starting point is $\ell_j^{H_0}$. \begin{equation}\label{eqn-diagram-j-h-epsilon}
    \begin{tikzcd}
    m_j^{H_0}& \arrow[l, mapsto, "\mathcal{I}^{-1}"'] \F (\ell_j^{H_0}) \\
    k_j^{H_0} \arrow[u,mapsto,"\HS"]   & \arrow[l,mapsto, "\A^{-1}"] \arrow[u,mapsto, "\F"']\ell_j^{H_0}
    \end{tikzcd}
    \end{equation}
    Lastly, and most importantly, it follows from Corollary \ref{cor-abel-support-invariance} (which uses the support preservation property of the Abel transform i.e., Theorem \ref{thm-abel-support-preserving}) that \begin{equation}\label{eqn-k-j-h-epsilon-enough}
        k_j(e^{H_0})=k_j^{H_0}(e^{H_0}),
    \end{equation} and thus, it is enough to estimate $|k_j^{H_0}(e^{H_0})|$ instead of $|k_j(e^{H_0})|$. 
\end{construction}
\textbf{Step 4: (Setting up the estimates)} It follows from Construction \ref{Chopping - 1} and Construction \ref{Chopping - 2} that to estimate $|k(e^{H_0})|$, it is enough to estimate the quantity \begin{equation}\label{eqn-sum-k-j-h-epsilon}
\sum_{j=0}^\infty |k_j^{H_0}(e^{H_0})|
\end{equation} since $\sum_{j=0}^\infty k_j$ converges to $k$ in the sense of distributions (\cite[Page 243]{Stein93-Real-Variable-Methods}). Using the definition of $k_j^{H_0}$ in the diagram \eqref{eqn-diagram-j-h-epsilon}, along with \eqref{eqn-kernel-expression-h0}, we obtain an estimate for $k_j^{H_0}$, given by  \begin{align}\label{eqn-bound-k-j-H-epsilon}
    \begin{split}
        |k_j^{H_0}(e^{H_0})|&= \frac{1}{|W|}\int_{\ab^*} |m_j^{H_0}(\lambda)|\cdot |\varphi_{-\lambda}(e^{H_0})|\,|\cfunc(\lambda)|^{-2}\,d\lambda\\
    &\leq \frac{\varphi_0(e^{H_0})}{|W|}\int_{\ab^*} |m^{H_0}_j(\lambda)|\cdot |\cfunc(\lambda)|^{-2}\,d\lambda\,.
    \end{split}
\end{align}  
The function $m_j^{H_0}$, unlike the function $m_j$ in \eqref{eqn-definition-j}, is no longer supported in the annulus $\{2^{j-1}\leq |\lambda|\leq 2^{j+1}\}$. For this reason, we will decompose the integral in the right-hand side of \eqref{eqn-bound-k-j-H-epsilon} as \begin{align}\label{eqn-setup-I-II}
    \begin{split}
    \int_{\ab^*} |m_j^{H_0}(\lambda)|\cdot &|\cfunc(\lambda)|^{-2}\,d\lambda \\&=\underbrace{\int_{|\lambda|\leq 2^{j+2}} |m_j^{H_0}(\lambda)|\cdot |\cfunc(\lambda)|^{-2}\,d\lambda}_{\eqqcolon \;\Circled{I}_j^{H_0}} + \underbrace{\int_{|\lambda|> 2^{j+2}} |m_j^{H_0}(\lambda)|\cdot |\cfunc(\lambda)|^{-2}\,d\lambda}_{\eqqcolon \;\Circled{II}_j^{H_0}}\,,
    \end{split}
\end{align} where the integrals $\Circled{I}_j^{H_0}$ and $\Circled{II}_j^{H_0}$ represent the contributions from the regions \emph{near} and \emph{far} from the original support of $m_j$, respectively. The rest of the proof will be about analyzing these integrals separately and then summing up over all $j$.\\

\textbf{Step 5: (The integral $\Circled{I}_j^{H_0}$)} To begin the analysis on $\Circled{I}_j^{H_0}$, an application of Cauchy-Schwarz inequality separates it into two parts i.e.,
\begin{align*}
    &\Circled{I}_j^{H_0}\;\;= \int_{|\lambda|\leq 2^{j+2}} |m_j^{H_0}(\lambda)|\cdot |\cfunc(\lambda)|^{-2}\,d\lambda\\&\;\,
    \overset{\text{Cauchy-Schwarz}}{\leq} \left(\int_{|\lambda|\leq 2^{j+2}} |m^{H_0}_j(\lambda)|^2\,d\lambda\right)^{\frac{1}{2}}\cdot \left(\int_{|\lambda|\leq 2^{j+2}} |\cfunc(\lambda)|^{-4}\,d\lambda\right)^{\frac{1}{2}}\,.
\end{align*} Now, the estimate for the Harish-Chandra $\cfunc$-function in \eqref{eqn-c-function-estimate} implies that

\begin{align}\label{eqn-first-bound-I-j}
    \begin{split}
    \Circled{I}_j^{H_0}
    &\lesssim \left(\int_{\ab^*} |m^{H_0}_j(\lambda)|^2\,d\lambda\right)^{\frac{1}{2}}\cdot 2^{j(n-\frac{a}{2})}\\&=\left(\int_{\ab^*} |(\mathcal{I}m^{H_0}_j)(\lambda)|^2\,d\lambda\right)^{\frac{1}{2}}\cdot 2^{j(n-\frac{a}{2})}
    \overset{\text{Plancherel}}{=} 2^{j(n-\frac{a}{2})}\cdot \|\ell^{H_0}_j\|_{L^2(\ab)}\,.
    \end{split}
\end{align} The last equality here follows from the Euclidean Plancherel formula on $\ab$. So, estimating $\Circled{I}_j^{H_0}$ boils down to estimating $\|\ell^{H_0}_j\|_{L^2(\ab)}$. The definition of $\ell^{H_0}_j$ in \eqref{eqn-definition-j-h-epsilon} and Lemma \ref{lem-l-j-bound} produce a bound for $\|\ell_j^{H_0}\|_{L^2(\ab)}$, which is given by \begin{align}\label{eqn-L2-bound-ell-H-epsilon}
    \begin{split}
    &\|\ell^{H_0}_j\|^2_{L^2(\ab)}\:\,=\int_{\ab} |\ell_j(H)|^2\cdot |\Omega^{H_0}(H)|^2\,dH\\&\quad \quad\;\,\overset{\text{Support of }\Omega^{H_0}}{\leq} \int_{|H|>\frac{|H_0|}{2}}|\ell_j(H)|^2\,dH\\&\quad \quad \quad \,\overset{\text{Lemma }\ref{lem-l-j-bound}}{\lesssim_{M}} 2^{2j(a+r-M)}\int_{|H|>\frac{|H_0|}{2}} |H|^{-2M}\,dH\\&\quad \quad \quad \quad\;\;\lesssim_M 2^{2j(a+r-M)}\cdot |H_0|^{-2M+{a}}
    \end{split}
\end{align} for all $M>\frac{a}{2}$. After taking a square root in the right-hand side of \eqref{eqn-L2-bound-ell-H-epsilon}, and combining it into \eqref{eqn-first-bound-I-j}, the bound on the $j$-th term $\Circled{I}_j^{H_0}$ is \begin{align}\label{eqn-final-bound-I-j}
    \begin{split}
    \Circled{I}_j^{H_0}&\lesssim_{M} |H_0|^{-M+\frac{a}{2}}\cdot 2^{j(n-\frac{a}{2})}\cdot 2^{j(a+r-M)}=|H_0|^{-M+\frac{a}{2}}\cdot 2^{j(n+r-M+\frac{a}{2})}
    \end{split}
\end{align} for all $M>\frac{a}{2}$.\\

\textbf{Step 6: (The integral $\Circled{II}_j^{H_0}$)} The analysis of this integral depends on the properties of the derivatives of the cut-off function $\omega^{H_0}$ in \eqref{eqn-derivative-pointwise-bounds-cutoff-H-epsilon} and their supports. A first step is the following lemma.
\begin{lemma}\label{lemma-bound-fourier-transform-cutoff-H-epsilon}
    For all $M'\geq 0$, there exists a constant $C_{M'}>0$ such that $$|\check{\omega}^{H_0}(\lambda)|\leq C_{M'}\cdot|H_0|^{-M'+a}\cdot \langle \lambda \rangle^{-M'}$$ for all $\lambda\in \ab^*$ such that $|\lambda|\geq 1$.
\end{lemma}
\begin{proof}
    Using the bounds in \eqref{eqn-derivative-pointwise-bounds-cutoff-H-epsilon}, and the support of the derivatives of $\omega^{H_0}$, this follows from an application of integration by parts for integer $M'\geq 0$, and then an interpolation argument shows it for all $M'\in \mathbb{R}_{\geq 0}$.
\end{proof}
By the definition of $\Circled{II}_j^{H_0}$, we need to estimate $m_j^{H_0}=m_j-m_j*\check{\omega}^{H_0}$ in the region $\{|\lambda|> 2^{j+2}\}$. However, the support condition on $\zeta_j$ (as in Construction \ref{Chopping - 1}) implies that $m_j=m\cdot \zeta_j$ is supported in $\{2^{j-1}\leq|\lambda|\leq 2^{j+1}\}$, and hence, $m_j$ vanishes in the region $\{|\lambda|> 2^{j+2}\}$. So, \begin{equation}
    m_j^{H_0}\equiv -m_j*\check{\omega}^{H_0}\quad \text{ on }\quad\{|\lambda|> 2^{j+2}\}\,.
\end{equation} Then, an application of Fubini's theorem, together with the supports of the cut-off functions $\zeta_j$ (as in Construction \ref{Chopping - 1}), implies that \begin{align}\label{eqn-preliminary-bound-II-j}
    \begin{split}
    \Circled{II}_j^{H_0}&\;\;= \int_{|\lambda|> 2^{j+2}} |m_j^{H_0}(\lambda)|\cdot |\cfunc(\lambda)|^{-2}\,d\lambda\\&\;\;= \int_{|\lambda|> 2^{j+2}} |m_j*\check{\omega}^{H_0}(\lambda)|\cdot |\cfunc(\lambda)|^{-2}\,d\lambda\\
    &\overset{\text{Fubini}}{\leq} \int_{\ab^*} |m_j(\mu)|\cdot \left(\int_{|\lambda|> 2^{j+2}} |\check{\omega}^{H_0}(\lambda-\mu)|\cdot |\cfunc(\lambda)|^{-2}\,d\lambda\right)\,d\mu\\
    &\;\;\leq \int_{2^{j-1}\leq|\mu|\leq 2^{j+1}} |m(\mu)|\cdot \left(\int_{|\lambda|> 2^{j+2}} |\check{\omega}^{H_0}(\lambda-\mu)|\cdot |\cfunc(\lambda)|^{-2}\,d\lambda\right)\,d\mu\,.
    \end{split}
\end{align}
The inner integral in the last line of \eqref{eqn-preliminary-bound-II-j} can be simplified by the change of variable $\eta=\lambda-\mu$, and this changes the limits of integration to $|\eta|\geq |\lambda|-|\mu|>2^{j+2}-2^{j+1}=2^{j+1}$. With this, the inner integral becomes \begin{align}\label{eqn-first-bound-II-j}
    \begin{split}
    &\int_{|\lambda|> 2^{j+2}} |\check{\omega}^{H_0}(\lambda-\mu)|\cdot |\cfunc(\lambda)|^{-2}\,d\lambda= \int_{|\eta|> 2^{j+1}} |\check{\omega}^{H_0}(\eta)|\cdot |\cfunc(\eta+\mu)|^{-2}\,d\eta\\
    &\quad \lesssim \int_{|\eta|> 2^{j+1}} |\check{\omega}^{H_0}(\eta)|\cdot \langle \eta+\mu\rangle^{n-a}\,d\eta\lesssim \int_{|\eta|> 2^{j+1}} |\check{\omega}^{H_0}(\eta)|\cdot \langle \eta\rangle^{n-a}\,d\eta\,.
    \end{split}
\end{align} The last inequality in \eqref{eqn-first-bound-II-j} uses the fact that $\langle \mu \rangle\lesssim 2^{j+1}\lesssim \langle \eta\rangle$, and hence, $\langle \eta+\mu\rangle\lesssim \langle \eta\rangle +\langle \mu \rangle \lesssim \langle \eta\rangle$. In particular, the estimate in \eqref{eqn-first-bound-II-j} is independent of $\mu$, so the $\mu$-integral in \eqref{eqn-preliminary-bound-II-j} separates out. It can then be estimated by the bound on the Harish-Chandra symbol $m$ in Definition \ref{defn-harish-chandra-symbol}, as\begin{equation}\label{eqn-bound-mu-integral}
\int_{2^{j-1}\leq|\mu|\leq 2^{j+1}} |m(\mu)|\,d\mu \lesssim 2^{j(r+a)}\,.
\end{equation} So, applying the estimates of Lemma \ref{lemma-bound-fourier-transform-cutoff-H-epsilon} into \eqref{eqn-first-bound-II-j}, and putting \eqref{eqn-first-bound-II-j} and \eqref{eqn-bound-mu-integral} into \eqref{eqn-preliminary-bound-II-j}, the bound for the $j$-th term $\Circled{II}_j^{H_0}$ becomes \begin{align}\label{eqn-bound-II-j-final}
    \begin{split}
    \Circled{II}_j^{H_0}&\lesssim_{M'}|H_0|^{-M'+a}\cdot 2^{j(r+a)}\int_{|\eta|\geq 2^{j+1}} \langle \eta\rangle^{n-a-M'}\,d\eta\\&
    \lesssim_{M'}|H_0|^{-M'+a}\cdot 2^{j(n-M')}=|H_0|^{-M'+a}\cdot 2^{j(n+r-M'+a)}
    \end{split}
\end{align} for all $M'>n$.\\

\textbf{Step 7: (Summing over $\Circled{I}_j^{H_0}$)} The argument here is taken from the proof of \cite[Chapter VI.4, Proposition 1]{Stein93-Real-Variable-Methods}. It follows from \eqref{eqn-final-bound-I-j} that \begin{equation}\label{eqn-sum-bound-I-j-H-epsilon}
    \sum_{j=0}^\infty \Circled{I}_j^{H_0}\lesssim_{M} |H_0|^{-M+\frac{a}{2}}\cdot \sum_{j=0}^\infty 2^{j(n+r-M+\frac{a}{2})}\,.
\end{equation} 
Summing over $j$ in \eqref{eqn-sum-bound-I-j-H-epsilon} involves separating it into \emph{two} different cases --- $|H_0|\leq 1$ and $|H_0|> 1$. 

Let us first \emph{assume} that $|H_0|\leq 1$. For this fixed $H_0$, we will divide the sum into \emph{two} pieces: one for $2^j\leq |H_0|^{-1}$ and another for $2^j > |H_0|^{-1}$ i.e., \begin{equation}\label{eqn-sum-divide-into-two}
\sum_{j=0}^\infty \Circled{I}_j^{H_0}\;=\;\sum_{2^j\leq |H_0|^{-1}} \Circled{I}_j^{H_0} + \sum_{2^j > |H_0|^{-1}} \Circled{I}_j^{H_0}\quad \text{ for }|H_0|\leq 1\,.
\end{equation} For the \emph{first} sum in \eqref{eqn-sum-divide-into-two}, \emph{fix} $M$ such that $0<M-\frac{a}{2}<1$ (remember the constraint $M>\frac{a}{2}$ from \eqref{eqn-L2-bound-ell-H-epsilon}), and hence, we need to estimate the following sum: \begin{align}\label{eqn-bound-near-field-1-preliminary}
    \begin{split}
        \sum_{2^j\leq |H_0|^{-1}} \Circled{I}_j^{H_0}&\lesssim  |H_0|^{-M+\frac{a}{2}}\cdot \sum_{j: 2^j\leq |H_0|^{-1}} 2^{j(n+r-M+\frac{a}{2})}\,.
    \end{split}
\end{align} Since $0<M-\frac{a}{2}<1$, the exponent $n+r-M+\frac{a}{2}$ in \eqref{eqn-bound-near-field-1-preliminary} is strictly positive (since $r>-a$ and $n>a$), and hence, the sum over $j$ in \eqref{eqn-bound-near-field-1-preliminary} can be bounded using the (finite) geometric series formula i.e., \begin{equation}\label{eqn-sum-bound-less-than-H-0-inverse}
    \sum_{j: 2^j\leq |H_0|^{-1}} 2^{j(n+r-M+\frac{a}{2})}\lesssim |H_0|^{-n-r+M-\frac{a}{2}}\overset{\left(\because\,|H_0|\leq 1\right)}{\lesssim_{N}} |H_0|^{-n-r-N+M-\frac{a}{2}}
\end{equation} for all $N\geq 0$. By using the bound in \eqref{eqn-sum-bound-less-than-H-0-inverse} and applying it to \eqref{eqn-bound-near-field-1-preliminary}, we obtain \begin{equation}\label{eqn-bound-near-field-1}
    \sum_{2^j\leq |H_0|^{-1}} \Circled{I}_j^{H_0}\lesssim_{N} |H_0|^{-n-r-N}
\end{equation}as an estimate for the first sum in \eqref{eqn-sum-divide-into-two} for all $N\geq 0$. For the \emph{second} sum in \eqref{eqn-sum-divide-into-two}, we will \emph{fix} $M>n+r+\frac{a}{2}$, and the sum over $j$ is bounded using the (infinite) geometric series formula i.e., \begin{align}\label{eqn-bound-near-field-2}
    \begin{split}
        \sum_{2^j > |H_0|^{-1}} \Circled{I}_j^{H_0}&\lesssim|H_0|^{-M+\frac{a}{2}}\cdot\sum_{2^j > |H_0|^{-1}} 2^{j(n+r-M+\frac{a}{2})}\\&\lesssim |H_0|^{-n-r+M-\frac{a}{2}}\cdot|H_0|^{-M+\frac{a}{2}}\;= \;|H_0|^{-n-r}\; {\lesssim_{N}}\; |H_0|^{-n-r-N}
    \end{split}
\end{align} for all $N\geq 0$. This completes the estimates for the sum over $j$ of $\Circled{I}_j^{H_0}$ for $|H_0|\leq 1$. 

Now, let us turn our attention to the \emph{second} case when $|H_0|>1$. \emph{Fix} $M>0$ such that $M-\frac{a}{2}>n+r+N$ for any $N\geq 0$. Almost identically, using the (infinite) geometric series formula, one recovers the same estimates as in \eqref{eqn-bound-near-field-2}. So, the two cases combine to provide the bound \begin{equation}\label{eqn-bound-sum-near-field-final}
    \sum_{j=0}^\infty \Circled{I}_j^{H_0}\lesssim_{N} |H_0|^{-n-r-N}\quad \quad \text{for all }H_0\in \overline{\ab^+}\setminus\{0\},
\end{equation} and for all $N\geq 0$.\\

\textbf{Step 8: (Summing over $\Circled{II}_j^{H_0}$)} The methods required here are the same as in Step 7. It follows from \eqref{eqn-bound-II-j-final} that the sum over $j$ of $\Circled{II}_j^{H_0}$ can be bounded by \begin{equation}\label{eqn-sum-bound-II-j-H-epsilon}
    \sum_{j=0}^\infty \Circled{II}_j^{H_0}\lesssim_{M'} |H_0|^{-M'+a}\cdot \sum_{j=0}^\infty 2^{j(n+r-M'+a)}\,,
\end{equation} which is essentially the same as the sum over $j$ of the bounds on $\Circled{I}_j^{H_0}$ in \eqref{eqn-sum-bound-I-j-H-epsilon}. We can follow the same \emph{two} cases as in Step 7 to bound the sum over $j$ of $\Circled{II}_j^{H_0}$. The only difference is that, to perform a similar analysis to \eqref{eqn-sum-bound-less-than-H-0-inverse} while estimating the sum $\sum_{2^j\leq |H_0|^{-1}}\Circled{II}_j^{H_0}$ for $|H_0|\leq 1$, we need the exponent in \eqref{eqn-sum-bound-II-j-H-epsilon} to satisfy $n+r-M'+a>0$. This can be done by choosing $M'>n$ such that $r+a>M'-n$, which is always possible because of our assumption $r>-a$. The remaining two sums proceed identically to their counterparts in Step 7. After combining the results from the two cases, the final outcome is that \begin{equation}\label{eqn-bound-sum-far-field-final}
    \sum_{j=0}^\infty \Circled{II}_j^{H_0}\lesssim_{N} |H_0|^{-n-r-N}\quad \quad \text{for all }H_0\in \overline{\ab^+}\setminus\{0\},
\end{equation} and for all $N\geq 0$.\\

\textbf{Step 9: (Final assembly for Harish-Chandra PSDOs of order $r>-a$)} As we stated after the reductions in Step 1, the goal was to estimate $|k(e^{H_0})|$ for any non-zero $H_0\in \overline{\ab^+}$ as per the expression in \eqref{eqn-kernel-expression-h0}, whenever $r>-a$. So,  combining the results from all the previous steps, we obtain \begin{align}\label{eqn-final-bound-k-h0}
    \begin{split}
    |k(e^{H_0})|&\;\;\;\leq\;\; \sum_{j=0}^\infty |k_j(e^{H_0})|\overset{\eqref{eqn-k-j-h-epsilon-enough}}{=}\sum_{j=0}^\infty |k_j^{H_0}(e^{H_0})|\\&\overset{\eqref{eqn-setup-I-II}}{\leq} \varphi_0(e^{H_0})\cdot\left(\sum_{j=0}^\infty \Circled{I}_j^{H_0} + \sum_{j=0}^\infty \Circled{II}_j^{H_0}\right)\\
    &\;\;\;\lesssim_{N} \,\varphi_0(e^{H_0})\cdot|H_0|^{-n-r-N}
    \end{split}
\end{align}\vspace{-.001in} for all $N\geq 0$. The last inequality follows from \eqref{eqn-bound-sum-near-field-final} and \eqref{eqn-bound-sum-far-field-final}. This proves the conclusion of Theorem \ref{thm-kernel-estimates-G/K} for all equivariant Harish-Chandra pseudodifferential operators of order $r>-a$.\\

\textbf{Step 10: (Kernel estimates for equivariant Harish-Chandra PSDOs of order $r\leq -a$)} To complete the proof of Theorem \ref{thm-kernel-estimates-G/K}, we need an upgraded version of the Laplacian trick in Lemma \ref{lem-Laplacian-trick}.

\begin{lemma}\label{lem-Laplacian-trick-upgraded}
    There exists a one-parameter family of properly-supported equivariant operators $\{E_s\}_{s\in \R}$ on $G/K$ such that \begin{itemize}
        \item $E_0$ is the identity operator on $G/K$,
        \item $E_s$ is an elliptic Hörmander pseudodifferential operator of order $s$ for all $s\in \R$, and
        \item $E_{s_1}\circ E_{s_2}=E_{s_1+s_2}+R_{s_1,s_2}$, where $R_{s_1,s_2}$ is a properly-supported equivariant smoothing operator for all $s_1,s_2\in \R$.
    \end{itemize}
\end{lemma}

\begin{remark}
    Theorem \ref{thm-Hormander-implies-HC} implies that $E_s$ is an equivariant Harish-Chandra pseudodifferential operator of order $s$ for all $s\in \R$.
\end{remark}

\begin{proof}
    Let $\Delta$ be the Laplace-Beltrami operator on $G/K$. Using spectral theory and holomorphic functional calculus, one can define the operator $(-\Delta)^{s/2}$ for all $s\in \R$, which is an equivariant elliptic Hörmander pseudodifferential operator of order $s$. See \cite{Seeley-Complex-powers-elliptic-operator-1967} 
    for a treatment on compact manifolds. The same argument can be extended to the Laplace-Beltrami operator on $G/K$ after observing that $\Delta$ is essentially self-adjoint and the spectrum of the $-\Delta$ lies in $[|\rho|^2,\infty).$ Such an argument can be found in \cite{Kato-fractional-powers-linear-operators-1960}. \footnote{Moreover, it follows that the spherical functions, $\varphi_\lambda$'s are eigenfunctions of $\Delta^s$, and hence, the Harish-Chandra symbol of $\Delta^s$ (not properly-supported) is $(-1)^s(|\lambda|^2+|\rho|^2)^s.$} 

    We will pick a properly-supported Hörmander pseudodifferential operator $E_s$ such that $E_s$ is equal to $(-\Delta)^{s/2}$ modulo smoothing operators (see \cite[Proposition 18.1.22]{Hormander-III}). The properties of the family of operators $\{E_s\}_{s\in \R}$, listed in Lemma \ref{lem-Laplacian-trick-upgraded}, then follow from \cite[Theorem 18.1.24]{Hormander-III}.
\end{proof}

Let us \emph{fix} $s>0$ such that $-a<r+s<0$. So, Lemma \ref{lem-Laplacian-trick-upgraded} lets us write the operator $T$ as \begin{equation}\label{eqn-decomposition-T}
    T=E_{-s}\circ (E_s\circ T)+R_s\circ T\,,
\end{equation} where $R_s$ is an equivariant properly-supported smoothing operator. Firstly, thanks to Theorem \ref{thm-Hormander-implies-HC-upgrade}, the operator $R_s$, and hence, the operator $R_s\circ T$, has a Harish-Chandra symbol in $S^{-\infty,G}_{\text{HC}}(G/K)$, which as a Fréchet space is isomorphic to the Euclidean Schwartz space $\schwartz(\ab^*)$. Then, Theorem \ref{thm-Schwartz-algebra-characterization} implies that the equivariant kernel of $R_s\circ T$ is in $\HCschwartz(\GmodmodK)$. This was also the situation in Example \ref{example-convolution-schwartz}. 

Now, $E_s\circ T$ is an equivariant Harish-Chandra PSDO of order $r+s$. Since $r+s>-a$, the equivariant kernel of $E_s\circ T$ satisfies the estimates  in \eqref{eqn-final-bound-k-h0}, and so it also satisfies the estimates $(*_{r+s})$ in Theorem \ref{thm-kernel-estimates-G/K}. Furthermore, since $r+s<0$, Theorem \ref{thm-kernel-PSDO-equivalence-Euclidean} together with the definition of a Hörmander PSDO in Section \ref{subsec-Hormander-PSDO} implies that $E_s\circ T$ is an equivariant Hörmander PSDO. Now, we will decompose the operator $E_s\circ T$ into $E_s \circ T\coloneqq (E_s\circ T)_0 + (E_s\circ T)_\infty$ such that $(E_s\circ T)_0$ is a properly-supported equivariant Hörmander PSDO of order $r+s<0$, and the equivariant kernel of $(E_s\circ T)_\infty$ is in $\HCschwartz(\GmodmodK)$. So, the term $E_{-s}\circ (E_s\circ T)$ in \eqref{eqn-decomposition-T} can be written as \begin{equation}\label{eqn-decomposition-Q-s-T}
    E_{-s}\circ (E_s\circ T)=E_{-s}\circ (E_s\circ T)_0 + E_{-s}\circ (E_s\circ T)_\infty\,.
\end{equation} The operator $E_{-s}\circ (E_s\circ T)_0$ is a properly-supported equivariant Hörmander PSDO of order $r$, and hence, the equivariant kernel of $E_{-s}\circ (E_s\circ T)_0$ satisfies the estimates $(*_r)$ in Theorem \ref{thm-kernel-estimates-G/K} because of Theorem \ref{thm-kernel-estimates-Euclidean} \footnote{The absence of the term $\varphi_0(e^{H_0})$ in the estimates given by Theorem \ref{thm-kernel-estimates-Euclidean} is not a hindrance due to the proper-supportedness of $E_{-s}\circ (E_s\circ T)_0$.}. Lastly, Theorem \ref{thm-Schwartz-algebra-characterization} implies that the operator $E_{-s}\circ (E_s\circ T)_\infty$ admits an equivariant kernel in $\HCschwartz(\GmodmodK)$. Combining the kernel estimates for all the terms in the decomposition of $T$ in \eqref{eqn-decomposition-T} and \eqref{eqn-decomposition-Q-s-T}, we conclude that the equivariant kernel of $T$ satisfies the estimates $(*_r)$ in the statement of Theorem \ref{thm-kernel-estimates-G/K}. This completes the proof of Theorem \ref{thm-kernel-estimates-G/K}.
\qed

\bibliographystyle{alphaurl}
\bibliography{citation.bib}

\end{document}